\documentclass[11pt]{amsart}
\usepackage{color}
\usepackage{amsmath,amssymb,amsfonts,amsthm,bm}
\usepackage{graphicx}
\usepackage{float}
\usepackage{enumerate}
\usepackage{appendix}
  \usepackage[shortlabels]{enumitem}
\usepackage[numbers, sort&compress]{natbib}
 \usepackage[left=1.5in, right=1.5in]{geometry}
\def\R{\mathbb R}
\def\Z{\mathbb Z}
\def\T{\mathbb T}

\def\N{\mathbb N}

\def\cB{\mathcal B}
\def\cC{\mathcal C}

\def\cP{\mathcal P}
\def\cQ{\mathcal Q}
\def\d{\partial}

\def\t{\dot}

\def\Id{{\rm Id}}

\renewcommand\div{{\rm div}\,}
\renewcommand\lim{{\rm lim}\,}

\renewcommand\sup{{\rm sup}\,}
\renewcommand\inf{{\rm inf}\,}
\renewcommand\log{{\rm log}\,}
\renewcommand\det{{\rm det}\,}

\newcommand{\with}{\quad\!\hbox{with}\!\quad}
\newcommand{\andf}{\quad\!\hbox{and}\!\quad}
\newcommand{\Sum}{\displaystyle \sum}

\def\dr{\delta\!\rho}
\def\du{\delta\!u}
\def\dv{\delta\!v}
\def\dw{\delta\!w}
\def\dA{\delta\!A}
\def\dP{\delta\!P}
\def\dQ{\delta\!Q}
\def\dX{\delta\!X}

\newtheorem{theorem}{Theorem}[section]
 
 \newtheorem{lemma}[theorem]{Lemma}
 \newtheorem{proposition}[theorem]{Proposition}
 \theoremstyle{definition}
 \newtheorem{definition}[theorem]{Definition}
 \theoremstyle{remark}
 \newtheorem{remark}[theorem]{Remark}
 
 \numberwithin{equation}{section}

\newcommand{\wt}{\widetilde}
\newcommand{\eps}{\varepsilon}

\newcommand{\abs}[1]{\left\vert#1\right\vert}

\newcommand{\norm}[1]{\Vert#1\Vert}

\DeclareMathOperator*{\esssup}{ess\,sup\,}
\DeclareMathOperator*{\essinf}{ess\,inf\,}

\begin{document}

\title[Two-dimensional inhomogeneous   Navier-Stokes equations]
{Global well-posedness for 2D    inhomogeneous viscous flows with rough data
via dynamic interpolation}
\author{ Rapha\"el Danchin}
\begin{abstract} We consider the evolution of two-dimensional incompressible 
flows with variable  density, only bounded and bounded away from zero.
Assuming  that the initial velocity belongs to a  suitable critical subspace of $L^2,$
 we prove a global-in-time existence and 
stability  result for the initial  (boundary) value problem. 

Our proof relies on new time decay estimates   for  finite energy   weak  solutions and
on a `dynamic interpolation' argument.
 We show  that the constructed solutions have a uniformly  $C^1$ flow, which  
ensures the propagation of geometrical structures in the fluid and  guarantees
 that the Eulerian and  Lagrangian formulations
of the equations are equivalent.  By adopting this latter formulation, we 
establish  the uniqueness of the solutions 
for prescribed data, and  the continuity of the flow map in an energy-like functional framework. 

In contrast with prior works, our results hold true in the critical regularity setting 
 \emph{without any smallness assumption.} 
Our approach uses only elementary tools and applies indistinctly  to the cases  
where the fluid domain is the whole plane, a smooth two-dimensional  bounded domain or the  torus.
\end{abstract}
\date{}
\keywords{Critical regularity, uniqueness, global solutions, inhomogeneous Navier-Stokes equations, rough density}
\subjclass[2010]{35Q30, 76D03, 76D05}

\maketitle
\section*{Introduction}

A huge literature has been devoted to the mathematical analysis of the  
Navier-Stokes equations that govern the evolution of the velocity field $u=u(t,x)$ and 
pressure function $P=P(t,x)$ of homogeneous incompressible viscous flows
in a domain $\Omega$ of $\R^d.$  Recall that these equations read
 \begin{equation*}
\left\{\begin{aligned}
 &u_{t}+\div(u\otimes u)-\mu\Delta u+\nabla P=0&&\quad\hbox{in }\ \R_+\times\Omega,  \\
&\div u=0&&\quad\hbox{in }\ \R_+\times\Omega,\\
&u|_{t=0}=u_0&&\quad\hbox{in }\Omega,
\end{aligned}\right.\leqno(NS)\end{equation*}
and, if $\Omega$ has a boundary, are supplemented with homogeneous Dirichlet boundary conditions for the velocity.
\smallbreak 
The global existence theory  for (NS) originates from  the 
paper \cite{Leray2} by J. Leray in 1934.  In the case $\Omega=\R^3,$ by combining   the   energy balance associated to (NS):
\begin{equation}\label{eq:L2NS}
\frac12\|u(t)\|_{L^2}^2+\mu\int_0^t\|\nabla u\|_{L^2}^2\,d\tau =  \frac12 \|u_0\|_{L^2}^2,
\end{equation}
with compactness arguments, he  constructed for any  divergence free 
$u_0$ in $L^2(\R^3;\R^3)$    a global  distributional solution of (NS) 
satisfying \eqref{eq:L2NS} \emph{with an inequality} (viz. the left-hand side is bounded by the right-hand side). 
\smallbreak
It is by now well understood that Leray's result is true  in any open subset $\Omega$ of $\R^d$ with $d=2,3$
(see for instance the first part of \cite{CDGG}). 
However, despite the numerous papers devoted to the topics   and  significant recent progresses, the question of uniqueness of finite
energy solutions in the case  $d=3$ has not been completely solved yet.  
The two-dimensional situation is much better understood:  finite energy solutions 
are unique and do satisfy \eqref{eq:L2NS} with an equality. 
Although uniqueness in dimension two could be hinted from another paper by J. Leray \cite{Leray} in 1934, 
it has been  established only  in 1959  by O.A. Ladyzhenskaya  \cite{Lady}, and  J.-L. Lions and G.~Prodi  \cite{LP}. 
\medbreak
In the present paper, we are concerned with \emph{inhomogeneous}, 
that is,  with variable density, incompressible viscous flows. The evolution of  these flows that 
can be encountered  in models of geophysics or mixtures, is   
often described by  the following \emph{inhomogeneous incompressible Navier-Stokes equations}:
 \begin{equation*}
\left\{\begin{aligned}
&\rho_{t}+\div(\rho u)=0\quad&\hbox{in }\ \R_+\times\Omega,  \\
 &(\rho u)_{t}+\div(\rho u\otimes u)-\mu\Delta u+\nabla P=0\quad&\hbox{in }\ \R_+\times\Omega,  \\
&\div u=0\quad&\hbox{in }\ \R_+\times\Omega. 
\end{aligned}\right.\leqno(INS)
\end{equation*}
Above, $u$ and $P$ still denote the velocity and the pressure, respectively, and $\rho=\rho(t,x)$
stands for the density that for obvious physical reasons  has to be  nonnegative. 
If we supplement (INS) with initial data and boundary conditions: 
\begin{equation}\label{eq:bc}
\rho|_{t=0}=\rho_0,\quad  u|_{t=0}=u_0\andf u|_{\partial\Omega}=0,\end{equation}
then  the energy balance  associated to (INS) reads: 
\begin{equation}\label{eq:L2INS}
\frac12\|(\sqrt\rho\,u)(t)\|_{L^2}^2+\mu\int_0^t\|\nabla u\|_{L^2}^2\,d\tau =  \frac12 \|\sqrt{\rho_0}\,u_0\|_{L^2}^2.
\end{equation}
The divergence free condition ensures that  the Lebesgue norms of $\rho$ are conserved, 
  and that 
\begin{equation}\label{eq:rho} \forall t\in\R_+,\quad
\underset{x\in\Omega}{\inf}\rho(t,x) =\underset{x\in\Omega}{\inf}\rho_0(x) 
\andf  \underset{x\in\Omega}{\sup}\rho(t,x) =\underset{x\in\Omega}{\sup}\rho_0(x).
\end{equation}
In the torus case, we have in addition the conservation of total momentum: 
\begin{equation}\label{eq:rhou}
\int_{\T^2} (\rho u)(t,x)\,dx = \int_{\T^2} (\rho_0 u_0)(x)\,dx.
\end{equation} 
Like (NS), equations (INS) have a scaling invariance (if $\Omega$ is stable by dilation): 
they are  invariant for all $\lambda>0$ by the transform:
\begin{equation}\label{eq:scaling}
(\rho,u, P)(t,x) \leadsto(\rho, \lambda u,\lambda^2 P)(\lambda^2t,\lambda x).
\end{equation}
Although (INS) is of hyperbolic-parabolic type while (NS) is parabolic,  similar results  hold true
for the initial value  (or boundary value) problem. For instance: 
\begin{itemize}
\item  In any dimension  and provided $\rho_0$ is bounded and nonnegative, and 
$\sqrt{\rho_0}\,u_0$ is in $L^2,$ there exists a global weak solution satisfying \eqref{eq:L2INS} with 
inequality\footnote{First proved by A.V.  Kazhikov in \cite{AVK} if $\rho_0>0,$ then for general  
$\rho_0\geq0$ by J. Simon \cite{JS}. In \cite{PLL},  P.-L. Lions pointed out that the density is a renormalized solution of the mass equation,  and treated density dependent viscosity coefficients. He also considered
 unbounded densities.}.
\item Smooth enough data with density bounded and bounded away from zero generate 
 a unique local-in-time smooth solution, which  is global in the two-dimensional case,
or in higher dimension if the initial   velocity is small\footnote{First established by  O.A.  Ladyzhenskaya and V.A. Solonnikov in \cite{LS}.}.
\end{itemize}
In dimension two, the  quantities  that come into play
in the energy balance \eqref{eq:L2INS}  are scaling invariant in the sense of \eqref{eq:scaling}. 
However, unlike  the case with constant density,  it is not known whether 
 finite energy two-dimensional weak solutions with bounded density, albeit having critical regularity,  are unique.

In order to  explain the difference between the variable and constant density cases  
and to motivate the assumptions that  will be made in this paper, let us sketch the proof of the uniqueness 
of finite energy solutions for (NS)  in dimension two. 
Assume that we are given two solutions $(u,P)$ and $(\wt u,\wt P)$   pertaining 
to the same finite energy initial velocity $u_0.$ 
Then,   $\du:=\wt u-u$ and $\dP:=\wt P-P$ satisfy
$$\left\{\begin{aligned}
 &\du_{t}+\div(u\otimes\du)-\mu\Delta \du+\nabla \dP=-\div(\du\otimes\wt u)\quad&\hbox{in }\ \R_+\times\Omega,  \\
&\div \du=0\quad&\hbox{in }\ \R_+\times\Omega.
\end{aligned}\right.$$
Taking the $L^2(\Omega;\R^2)$ scalar product with $\du,$ integrating by parts where needed and using H\"older inequality 
to bound the right-hand side yields
$$
\frac12\frac d{dt}\|\du\|_{L^2}^2+\mu\|\nabla\du\|_{L^2}^2 \leq \|\nabla\wt u\|_{L^2}\|\du\|_{L^4}^2,
$$
which, in light of  the celebrated  Ladyzhenskaya  inequality
\begin{equation}\label{eq:lad}
\|z\|_{L^4}^2\leq C\|z\|_{L^2}\|\nabla z\|_{L^2}
\end{equation}
leads to 
$$\begin{aligned}
\frac12\frac d{dt}\|\du\|_{L^2}^2+\mu\|\nabla\du\|_{L^2}^2 &\leq C\|\nabla \wt u\|_{L^2}\|\du\|_{L^2}
\|\nabla\du\|_{L^2}\\&\leq\frac\mu2\|\nabla\du\|_{L^2}^2+\frac{C^2}{2\mu}\|\nabla \wt u\|_{L^2}^2\|\du\|_{L^2}^2.
\end{aligned}$$
At this stage, Gronwall lemma allows to conclude that 
$$\|\du(t)\|_{L^2}^2+\mu\int_0^t\|\nabla\du\|_{L^2}^2 \,d\tau\leq 
e^{\frac{C^2}{\mu}\int_0^t\|\nabla\wt u\|_{L^2}^2\,d\tau}\|\du(0)\|_{L^2}^2.$$
Owing to \eqref{eq:L2NS}, the exponential term if finite. Hence we have
 $\du\equiv0$ if $\wt u(0)=u(0).$
\medbreak
In contrast, when  comparing two finite energy  solutions $(\rho,u,P)$ and $(\wt\rho,\wt u,\wt P)$ of (INS),
we get the following system for  $\dr:=\wt\rho-\rho,$ $\du$ and $\dP$: 
$$\left\{\begin{aligned}
&\dr_t+\div(\dr\,u) =-\div(\wt\rho\,\du),\\
 &(\rho\du)_{t}+\div (\rho u\otimes \nabla \du)-\mu\Delta \du+\nabla \dP=-(\dr \,\wt u)_t-\div(\rho u\otimes  \du)
 -\div(\rho\du\otimes \wt u),  \\
&\div \du=0.
\end{aligned}\right.$$
Since $\wt\rho$ is only bounded, the first line is a transport equation by the divergence free vector-field $u,$ 
with a source term that has (at most)  the regularity $C^{-1}$ with respect to the space variable. 
Now, in order to control the propagation of negative regularity in a transport
equation, we need
\begin{equation}\label{eq:Lip}
\nabla u\in L^1_{loc}(\R_+;L^\infty).\end{equation}
However, this property generally fails for finite energy solutions of (INS) and even for the  two-dimensional heat equation. 
 In fact,   the set of functions $u_0$ so that the solution $u$ to the free heat equation with initial data $u_0$ satisfies 
  $\nabla u\in L^1(\R_+;L^\infty)$ is the homogeneous Besov space $\dot B^{-1}_{\infty,1},$ and  $L^2$ is  not  embedded
in this space. 

To avoid working in spaces with negative regularity, one can recast
(INS) in the Lagrangian coordinates system as in  \cite{DM1}. Then,  the density becomes time independent and 
the velocity equation keeps its parabolicity (at least for small time).  
However, the equivalence between the Eulerian and Lagrangian formulations of (INS) in our low regularity context
still requires \eqref{eq:Lip}, a property that   cannot be expected if $u_0$ is only in $L^2$
since it fails for  the heat flow.

To make a long story short,  it is not clear that uniqueness holds for (INS) in the framework of 
just finite energy solutions.
\medbreak
Before describing in more detail the main objective of the article, let us recall some recent results on the
 well-posedness theory for (INS). 
A number of works have been devoted to this issue  under weaker assumptions 
than in \cite{LS}.  This is mainly to relax the positivity condition on the density or the regularity assumptions on 
the initial data. 
Regarding  the first question, it has been observed by  Y. Cho and H. Kim in \cite{CK} 
that  (INS) is  well-posed  for  smooth enough data and, possibly, vanishing  densities satisfying a suitable compatibility condition. 
Recently, J. Li in \cite{JLi} discovered that this condition is no longer  needed if one  considers
$H^1$ regularity for the velocity, and  the full well-posedness theory for general only bounded (not necessarily positive) 
initial densities and $H^1$ velocities has been carried out in a joint work with P.B. Mucha~\cite{DM1}. 
\smallbreak
Regarding  the minimal regularity requirement of the velocity  for  well-posedness, the scaling invariance 
of (INS) pointed out in \eqref{eq:scaling}  suggests (if $\Omega=\R^d$)  to take
$\rho_0\in L^\infty(\R^d)$ and $u_0\in\dot H^{\frac d2-1}(\R^d).$
In the constant density case and for $d=3,$  this assumption is in accordance with the well-known Fujita and Kato theorem 
\cite{FK}.  However as, again, $\nabla e^{t\Delta} u_0$ need not be in $L^1_{loc}(\R_+;L^\infty)$
if  $u_0\in\dot H^{\frac d2-1}(\R^d)$ then  it is not clear that uniqueness may be achieved if no additional regularity, in the
variable density case.  In this direction, it has been proved  in  \cite{DR2003,DR2004} that if 
 $u_0$ belongs to the homogeneous Besov space $\dot B^{\frac d2-1}_{2,1}(\R^d),$ a large subspace of
 $\dot H^{\frac d2-1}(\R^d)$ with the same scaling invariance,   then (INS) 
 is globally well-posed in dimension two (or in higher dimension  if $u_0$  is small)
 \emph{provided   $\rho_0$ is close to some positive constant in the homogeneous Besov space $\dot B^{\frac d2}_{2,1}(\R^d).$}
  This result is satisfactory as regards the regularity requirement for the velocity, since it is critical and closely related 
 to the $L^2$ space,  but the condition  on  the density is  rather restrictive both 
 because $\rho_0$   has to be almost constant and since it has to be 
 continuous (the space  $\dot B^{\frac d2}_{2,1}(\R^d)$  is embedded in the set   $\cC_b(\R^d)$ of bounded and continuous functions on $\R^d$). 
The result of \cite{DR2003} has been significantly improved recently in the two-dimensional case:   H. Abidi and G. Gui \cite{AG} 
established the global well-posedness   without any smallness condition on the data   if $\rho_0-1$ is in $\dot B^{1}_{2,1}(\R^2)$  and 
$u_0$ belongs to $\dot B^0_{2,1}(\R^2).$   The corresponding result in dimension three has been 
obtained with completely different techniques by H. Xu in \cite{Xu} (for small $u_0$  of course).   As said before, 
works based on the use of critical  Besov spaces for the density precludes considering the case
of densities that are  discontinuous along an interface, a situation which   is of particular interest  if one 
believes (INS) to be a relevant  model for mixtures of incompressible viscous flows with different densities.
This very situation, that is sometimes called \emph{the density patch problem} has been extensively studied
lately, see e.g. \cite{LZ,DM1,GG}. 

Well-posedness results for only bounded  initial  density, bounded
away from zero, and  smooth enough  velocity  have been obtained  in a joint work with P.B. Mucha \cite{DM4}, 
then improved   by M. Paicu, P. Zhang and Z. Zhang in \cite{PZZ}
 (there, $u_0$ is in $H^s(\R^2)$ for some $s>0$ if $d=2,$
and in $H^1(\R^3)$ if $d=3$).  In the whole space case, the critical regularity index has been reached 
 in an intriguing work  by P.~Zhang  \cite{Zhang19}. He  established 
the global existence for any small enough divergence free $u_0$ with coefficients in $\dot B^{\frac12}_{2,1}(\R^3)$ 
while  $\rho_0$ is only  bounded and bounded away from zero. It has been 
observed recently  in a joint work with S. Wang \cite{DW} that Zhang's solutions actually 
satisfy \eqref{eq:Lip}, and are thus unique. 
\medbreak
The main  goal of the present paper is to investigate    the counterpart 
\emph{in dimension two and for large initial data} of P. Zhang's result recalled  just above:
we want to establish a global well-posedness result  for  general  divergence-free velocity fields $u_0$ with  critical regularity of $L^2$ type 
and   densities  $\rho_0$ just satisfying:
\begin{equation}\label{eq:rho00}
\rho_*:=\underset{x\in\Omega}{\essinf} \,\rho_0(x)>0\andf \rho^*:=\underset{x\in\Omega}{\esssup} \,\rho_0(x)<\infty.
\end{equation}
According to \cite{AG}, a good candidate to achieve the  Lipschitz property within a critical regularity framework of $L^2$ type is the space  $\dot B^0_{2,1}.$  However, owing to the use of  Fourier analysis techniques,   
rather strong regularity assumptions on the density were made in \cite{AG}. 
Here, since we want to consider only bounded densities,  we shall adopt a completely different approach. 
In fact,   we shall combine real interpolation and  three levels of time decay estimates  
(corresponding to $\dot H^{-1},$ $L^2$ and 
 $\dot H^1$ data, respectively) for a linearized version of (INS)  that can be obtained 
 just by energy arguments, and basic properties of the Stokes system, 
 so as to work out a space for $u_0$  that coincides with  $\dot B^0_{2,1}$
if $\rho_0$ is smooth (but that might depend on it if it is not).
The overall  strategy  is so robust that it can  be adapted  to other systems. 
\smallbreak
The rest of the paper is structured as follows:  in the next section 
we state our main results and explain the key steps of the proof. 
Then, in Section \ref{s:weak}, we establish  a first family of time decay estimates pertaining to the case where $u_0$ is just in $L^2,$
and construct  corresponding global finite energy  weak solutions for (INS). 
The next section is devoted to proving more a priori  decay estimates. The final goal is to establish that  under 
a slightly stronger assumption on the initial velocity, very close to the regularity $\dot B^0_{2,1},$ 
the Lipschitz property \eqref{eq:Lip} is satisfied. 
  Finally, we establish in Section \ref{s:proof} the existence and  uniqueness
of a solution under this assumption, assuming only \eqref{eq:rho00} 
 and that  the velocity belongs to the  aforementioned space.  
 The same method also provides stability estimates for the flow map, in the energy space. 
\medbreak
\noindent{\bf Notation:}  In the rest   of the paper, $\Omega$ will be either  a $C^2$ bounded domain of $\R^2,$ 
   a two-dimensional torus, or $\R^2.$  It will be convenient to use the  same notation $\dot H^s(\Omega)$
 to designate:
 \begin{itemize}
 \item[--]   the classical homogeneous Sobolev space if  $\Omega=\R^2,$ 
 \item[--]  the subset of functions of $H^s$ with mean value $0$ if $\Omega=\T^2,$
 \item[--]   the space $H^s_0(\Omega)$ (that is the  completion  of $\cC^\infty_c(\Omega)$ for the $H^s(\R^2)$ norm) if $\Omega$
 is a bounded domain and $s\in[0,1]$;
 \item[--]  the dual of $H^{-s}_0(\Omega)$   if $\Omega$  is a bounded domain and $s\in[-1,0].$
\end{itemize}
 We  designate  by    $L^2_\sigma(\Omega)$ the set of  divergence free vector-fields with  
  coefficients in $L^2(\Omega)$  (such that $u_0\cdot n=0$ at $\d\Omega$ in the bounded domain case, 
  with $n$ being the unit exterior normal vector to $\d\Omega$), and denote by $\cP$
  the orthogonal projector from $L^2(\Omega;\R^2)$ to $L^2_\sigma(\Omega).$

For any normed space $X,$  Lebesgue index $q\in[1,\infty]$ and time $T\in[0,\infty],$ 
we shall denote $\|z\|_{L^q_T(X)}:= \bigl\| \|z(t)\|_X\|_{L^q(0,T)}$  and omit $T$ if it is $\infty.$
In the case where $z$ has several components  in $X,$ we  keep the same notation for the norm. 

  As usual,  $C$ designates harmless positive real numbers, and we shall often write
$A\lesssim B$ instead of $A\leq CB.$ To emphasize the dependency with respect to parameters $a_1,\cdots,a_n,$
we adopt the notation $C_{a_1,\cdots,a_n}.$ 
The notation $C_{\rho,v}$ stands for various `constants'  that only depend (algebraically) on the infimum and supremum of $\rho$  and on `energy-like' norms of $v,$ that is,  \emph{on norms that could be eventually bounded by $\|u_0\|_{L^2}$ if $(\rho,v)$
were a solution to (INS).} Obvious examples  are  $\|v\|_{L^\infty(L^2)}$ or $\|\nabla v\|_{L^2(L^2)}$
(remember \eqref{eq:L2INS}) but also $\|v\|_{L^4(L^4)}$ (use \eqref{eq:lad}) and so on.

\medbreak\noindent{\bf Acknowledgments.}  The author is indebted 
to P. Auscher for clarifying some properties of the real interpolation space  in which 
the initial velocity is taken, and to the anonymous referee for insightful remarks.

 
 \section{Results and strategy} \label{s:results}

 The first step  is to exhibit  time decay estimates 
 for finite energy solutions. More precisely, we shall establish  the following  statement:
  \begin{theorem} \label{thm:1} 
   Let $u_0$ be in $L^2_\sigma(\Omega)$ 
    and $\rho_0$ satisfy \eqref{eq:rho00}. 
   Then, (INS) supplemented with \eqref{eq:bc}
admits a global solution $(\rho,u,P)$ satisfying \eqref{eq:rho} (and \eqref{eq:rhou} if $\Omega=\T^2$), 
 $u\in L^\infty(\R_+;L^2_\sigma),$ 
  $\nabla u\in L^2(\R_+\times\Omega),$  and 
  \begin{equation}\label{ineq:L2INS}
\frac12\|(\sqrt\rho\,u)(t)\|_{L^2}^2+\mu\int_0^t\|\nabla u\|_{L^2}^2\,d\tau \leq \frac12 \|\sqrt{\rho_0}\,u_0\|_{L^2}^2,\qquad t>0.
\end{equation}
Furthermore, there exists a constant $C$ depending only on  $\Omega,$ $\rho_*$ and $\rho^*$ such that for all $t>0,$ we have
$$\begin{aligned}
\|\nabla^k u(t)\|_{L^2} &\leq C(\mu t)^{-k/2} \|u_0\|_{L^2} \quad\hbox{for }\ k=0,1,2,\\
\|\nabla^k(u_t,\dot u)(t)\|_{L^2}   &\leq C(\mu t)^{-1-k/2} \|u_0\|_{L^2} \quad\hbox{for }\ k=0,1,\\
\|\nabla P(t)\|_{L^2}&\leq Ct^{-1}\|u_0\|_{L^2},
\end{aligned}$$
where $\dot u$ denotes the convective derivative of $u,$ that is, $\dot u:=u_t+u\cdot\nabla u.$
 \end{theorem}
 Two remarks are in order:
 \begin{itemize} 
  \item[--] The constructed solutions satisfy more time decay estimates : see \eqref{eq:weight3}, 
  \eqref{eq:weighttt}, \eqref{eq:weight3/2}, Proposition \ref{p:decay1} with $s'=0$ and  Proposition \ref{p:decay2} with $p=2.$
 \item[--]  As pointed out in \cite{DMP}  for $H^1_0(\Omega)$  initial velocities, 
 exponential time decay estimates hold true  if $\Omega$ is bounded.
 Following the proof of Lemma 5 therein, one can show that there exists a positive constant 
  $c_\Omega$ depending only on $\Omega$ such that 
   $$ \forall t\in\R_+,\; \|(\sqrt\rho\, u)(t)\|_{L^2} \leq e^{-c_\Omega\frac{\mu t}{\rho^*}}\|\sqrt{\rho_0}\, u_0\|_{L^2}\cdotp$$
   {}From this inequality, one can  deduce 
 exponential decay  for $\|t^{k/2}\nabla^k u\|_{L^2},$ $\|t^{1+k/2}\nabla^k u_t\|_{L^2}$  and $\|t^{1+k/2}\nabla^k\dot u\|_{L^2}.$
 However, as exponential decay does not  hold if $\Omega=\R^2,$ and since
 we strive for  a unified approach, we refrain from tracking it in the rest of the paper, 
 to simplify the presentation. 
 \end{itemize} 
  As underlined in the introduction, in order to establish the uniqueness of solutions,  we need 
 a  functional space  that ensures \eqref{eq:Lip}. 
 At the same time, we want our functional framework to be critical, to allow any initial 
 density just bounded and bounded away from zero and to be strongly related to the energy space  $L^2.$
 Note that Theorem \ref{thm:1} ensures that $\nabla u$ belongs to the \emph{weak} $L^1$ space for the time
 variable with values in the Sobolev space $H^1.$ This latter space `almost' embeds in $L^\infty.$ 
  A classical way to improve embeddings is to 
 work out a space by means of \emph{real interpolation with second  parameter equal to $1.$} In our context, 
 since energy arguments play an important role,  it is natural to interpolate from Sobolev spaces and to  consider\footnote{One could prefer to interpolate between \emph{Lebesgue spaces}
 and  consider the velocity in the Lorentz space $L^{2,1}.$  However 
  we do not know how to handle (INS) in this space. The reader~is referred to \cite{D-press}
  where the space  $L^{2,1}$ is used for solving the two-dimensional system for pressureless gases. } 
 \begin{equation}\label{eq:interpo} [\dot H^{-s},\dot H^{s}]_{1/2,1}\quad\hbox{for some}\quad s\in(0,1).\end{equation}
 This definition gives the Besov space $\dot B^0_{2,1}$ (independently of the value of $s$). 
\smallbreak
  Let us shortly explain  why  in the simpler situation where $u$  is the solution of the 
 free heat equation  in $\R^2,$ supplemented with an initial data $u_0$ in $\dot B^0_{2,1},$ we do have  \eqref{eq:Lip}. 
 We start from the following two inequalities:
 \begin{equation}\label{eq:heatineq}
 t\|\nabla u(t)\|_{L^\infty}\leq C\min\bigl(t^{s/2}\|u_0\|_{\dot H^s},\:t^{-s/2}\|u_0\|_{\dot H^{-s}}\bigr)
 \end{equation}
 which may be easily derived by using the explicit formula for $u$ in the Fourier space. 
 \medbreak
 Then, we use the characterization of real interpolation spaces in terms of atomic decomposition like in e.g. \cite{LPe}. 
 In our setting, it reads $z\in \dot B^{0}_{2,1}$ if and only if  
 there exists a sequence $(z_j)_{j\in\Z}$  of $\dot H^{-s}\cap \dot H^{s}$ satisfying:
 $$z=\sum_{j\in\Z} z_j\andf \sum_{j\in\Z}\bigl(2^{-j/2} \|z_j\|_{\dot H^{s}} + 2^{j/2} \| z_j\|_{\dot H^{-s}}\bigr)<\infty.$$
The infimum  of the above sum on all admissible decompositions of $z$ defines a norm on~$\dot B^0_{2,1}.$ 
 Now, decompose $u_0$ into 
 \begin{equation}\label{eq:decompotriviale0}
 u_0=\sum_{j\in\Z} u_{0,j}\with \sum_{j\in\Z}\bigl(2^{-j/2} \|u_{0,j}\|_{\dot H^{s}} 
 + 2^{j/2} \|u_{0,j}\|_{\dot H^{-s}}\bigr)\leq 2\|u_0\|_{\dot B^0_{2,1}}\end{equation}
and solve all the heat equations
$$(u_j)_t-\Delta u_j=0,\qquad u_j|_{t=0}=u_{0,j}.$$ 
As the heat equation is linear, we have
$u=\sum_j u_j$ and thus
\begin{equation}\label{eq:decompotriviale}\int_0^\infty\|\nabla u\|_{L^\infty}\,dt\leq \sum_{j\in\Z}\int_0^\infty\|\nabla u_j\|_{L^\infty}\,dt.
\end{equation}
Now, for every $j$ in $\Z$  and $A_j>0,$ we have, due to \eqref{eq:heatineq},
$$\begin{aligned}
\int_0^\infty\|\nabla u_j\|_{L^\infty}\,dt&\leq \int_0^{A_j}\|\nabla u_j\|_{L^\infty}\,dt+\int_{A_j}^\infty\|\nabla u_j\|_{L^\infty}\,dt\\
&\lesssim \|u_{0,j}\|_{\dot H^s}\int_0^{A_j} t^{-1+s/2}\,dt+\|u_{0,j}\|_{\dot H^{-s}}
\int_{A_j}^\infty  t^{-1-s/2}\,dt\\
&\lesssim   \|u_{0,j}\|_{\dot H^s} A_j^{s/2}+\|u_{0,j}\|_{\dot H^{-s}}A_j^{-s/2}. \end{aligned}$$
Hence, choosing $A_j=2^{-j/s}$ and remembering    \eqref{eq:decompotriviale0} and \eqref{eq:decompotriviale}
gives \eqref{eq:Lip} (globally in time). 
 
 This  `dynamic interpolation approach' has been used before by T. Hmidi and S. Keraani in \cite{HK} for the transport
 equation and  by P. Zhang in \cite{Zhang19} for the velocity equation of (INS) (in dimension $3$ and for small velocities). In both 
 cases however, the initial data was decomposed according to a Littlewood-Paley decomposition. The additional flexibility 
 that consists here in using general atomic decompositions enables us  to do without Fourier analysis and to treat general domains.  
 
As our aim is to prove \eqref{eq:Lip} for (INS), we have to consider instead of the heat equation a  linear system which captures both the effects of the density and of the convection. 
To this end, we consider  
 \begin{equation}\label{eq:LINS0}
\left\{\begin{aligned}
 &(\rho u)_{t}+\div(v\otimes  u)- \Delta u+\nabla P=0\quad&\hbox{in }\ \R_+\times\Omega, \\
&\div u=0\quad&\hbox{in }\ \R_+\times\Omega,\\
&u|_{t=0}=u_{0}\quad&\hbox{in }\ \Omega,
\end{aligned}\right.
\end{equation}
where the (smooth enough) triplet   $(\rho,v,u_0)$ is  given  with $\rho$ bounded and bounded away from zero, 
\begin{equation}\label{eq:condrhov}
\rho_{t}+\div (\rho v)=0,\quad \div v=0\andf v|_{\partial\Omega}=0.
\end{equation}
Clearly, if we succeed in proving \eqref{eq:heatineq} for \eqref{eq:LINS0}  with a constant that only 
depends on $\rho_*,$ $\rho^*$ and of energy-like norms of $v,$ then repeating the above dynamic interpolation 
procedure will yield  \eqref{eq:Lip} for the solutions of \eqref{eq:LINS0} supplemented with initial data
in $\dot B^0_{2,1},$ then for (INS) if taking $v=u.$  

The way to get \eqref{eq:heatineq} is to prove beforehand three families of time 
weighted estimates for \eqref{eq:LINS0} corresponding  to initial data $u_0$ in $L^2,$
$\dot H^1$ and $\dot H^{-1},$ respectively.   The estimate in $\dot H^{-1}$
will be obtained by duality from the  estimate in $\dot H^1.$  This will lead us to consider  
the  backward  system associated with  \eqref{eq:LINS0} and  it  is rather  $\|\cP(\rho u)(t)\|_{\dot H^{-1}}$
and, more generally,  $\|\cP(\rho u)(t)\|_{\dot H^{-s}}$ for $s\in (0,1)$ that can be estimated. 
In the end,  combining the three families of inequalities with suitable Gagliardo-Nirenberg inequalities yields
 instead of \eqref{eq:heatineq}, 
\begin{equation}\label{eq:INSineq}
 t\|\nabla u(t)\|_{L^\infty}\leq C_{\rho,v} \min\bigl(t^{s/2}\|u_0\|_{\dot H^s},\, t^{-s/2}\|\cP(\rho_0 u_0)\|_{\dot H^{-s}}\bigr)\cdotp
 \end{equation}
Above, $C_{\rho,v}$ only depends on $\rho_*,$ $\rho^*$ and on energy-like norms of $v.$ 
\medbreak
As a consequence, the suitable  interpolation space  to carry out our  dynamic interpolation
procedure for \eqref{eq:LINS0}  is the one that is given in the following definition: 
\begin{definition}\label{def:espacepourri} Let $s$ be in $(0,1)$ and $a$ be a   measurable function  on $\Omega$ with positive
lower bound.
We denote by    $\wt B^{0,s}_{a,1}(\Omega)$  the set  of   vector-fields $z$  in  $L^2_\sigma(\Omega)$  
   such that there exists a sequence $(z_j)_{j\in\Z}$  of $L^2_\sigma(\Omega)$ satisfying:
\begin{itemize}
\item[---] $z=\sum_{j\in\Z} z_j$ in the sense of distributions,
\item[---] for all $j\in\Z,$ there holds  $\cP(az_j)\in\dot H^{-s}(\Omega)$ and $z_j\in\dot H^{s}(\Omega),$
\item[---]  $\sum_{j\in\Z}\bigl(2^{-j/2} \|z_j\|_{\dot H^{s}} + 2^{j/2} \|\cP(a z_j)\|_{\dot H^{-s}}\bigr)$ is finite.
\end{itemize}
The infimum  on all admissible decompositions of $z$ defines a norm on~$\wt B^{0,s}_{a,1}(\Omega).$ 
\end{definition}
Let us highlight a few properties of these spaces.
\begin{itemize}
\item The family  $(\wt B^{0,s}_{a,1}(\Omega))_{s\in(0,1)}$ is a family of nested Banach spaces:
if $0<s'<s<1,$ then $\wt B^{0,s}_{a,1}(\Omega)\hookrightarrow \wt B^{0,s'}_{a,1}(\Omega).$
\item 
Owing to \eqref{eq:interpo},  if  $a$ is a positive constant,
 then $\wt B^{0,s}_{a,1}$  is nothing than $\dot B^0_{2,1},$  and if  $a$ has a positive lower bound $a_*,$  then 
 it embedded in $L^2.$  
  Indeed,  decomposing $z\in\wt B^{0,s}_{a,1}$  according to Definition \ref{def:espacepourri}
  and using the fact that $\cP$ is a $L^2$ orthogonal projector, one may write for all $j\in\Z,$
   \begin{equation}\label{eq:L2zj}
    \|z_j\|_{L^2}^2\leq a_*^{-1} \int_\Omega \cP(az_j) \cdot z_j\,dx \leq  a_*^{-1}\bigl(2^{j/2} \|\cP(a z_j)\|_{\dot H^{-1/2}}\bigr)
  \bigl(2^{-j/2} \|z_j\|_{\dot H^{1/2}}\bigr),\end{equation}
 which implies, by Young inequality,  that 
 $$ \|z\|_{L^2}\leq \frac1{2\sqrt{a_*}}\,\|z\|_{\wt B^{0,s}_{a,1}}.$$
 \item 
If  $a$ is  bounded and $s=2/p-1$ for some $p\in(1,2),$ then the critical Besov space  
$\dot B^{-1+2/p}_{p,1}:= [L^{p},\dot W^{2s}_p]_{1/2,1}$
 is embedded in $\wt B^{0,s}_{a,1}.$ Indeed,    if $z\in \dot B^{-1+2/p}_{p,1},$ then there exists a sequence
 $(z_j)_{j\in\Z}$ of the nonhomogeneous Sobolev space  $W^{2s}_p$ such that
 $$ z=\sum_{j\in\Z} z_j\andf \sum_{j\in\Z} \bigl(2^{-j/2}\|z_j\|_{W^{2s}_{p}} +2^{j/2}\|z_j\|_{L^{p}}\bigr)\leq 
 2\|z\|_{\dot B^{-1+2/p}_{p,1}}.$$
Now, the fact that $\cP:L^p\to L^p,$ and the  embeddings $\dot W^{2s}_p\hookrightarrow\dot H^{s}$ and 
$L^{p}\hookrightarrow \dot H^{-s}$  allow to write that
$$ \|z_j\|_{\dot H^{s}}\leq C\|z_j\|_{\dot W^{2s}_p}\andf
 \|\cP(az_j)\|_{\dot H^{-s}}\leq C\|\cP(a z_j)\|_{L^{p}}\leq C\|a\|_{L^\infty} \|z_j\|_{L^{p}},$$
 which gives our claim.
 \item  For general measurable functions $a$ bounded and bounded away from zero, 
 the space $\wt B^{0,s}_{a,1}$ might  depend on $s.$  However, in the case $s\in(0,1/2),$ 
if   $a$ is positive and piecewise constant along a finite number of Lipschitz  curves, then  it coincides with $\dot B^0_{2,1}.$
Indeed, in this case  the space $\dot H^{-s}$ is stable by multiplication by piecewise constant functions.
  \end{itemize}
 Our main global 
  existence and uniqueness statement reads as follows:
    \begin{theorem} \label{thm:2}  Let  $\rho_0$  satisfy \eqref{eq:rho00}  and $u_0$ be  in $\wt B^{0,s}_{\rho_0,1}$
    for some $s\in(0,1).$ 
    Then, (INS) supplemented with \eqref{eq:bc} admits a unique global solution $(\rho,u,\nabla P)$
 satisfying all the properties stated in Theorem \ref{thm:1} (and the remarks that follow) and the energy balance 
  \eqref{eq:L2INS}.  In addition, we have 
 $$u\in\cC(\R_+;L^2),\quad
 \nabla u\in L^1(\R_+;C_b\cap \dot H^1),\quad \sqrt t(\dot u,\nabla P,\nabla^2u)\in L^{4/3}(\R_+;L^4)$$ 
   and, for all $t\in\R_+,$
 we have $u(t)\in \wt B^{0,s}_{\rho(t),1}$ with the inequality 
 \begin{equation}\label{eq:utB} \|u(t)\|_{\wt B^{0,s}_{\rho(t),1}}\leq C \|u_0\|_{\wt B^{0,s}_{\rho_0,1}}.\end{equation}
 \end{theorem} 
 \begin{remark} As a by-product of the proof of the uniqueness, we get a stability result  with respect to the  initial data 
  in the energy space  (see  Theorem \ref{thm:3} below). 
 \end{remark} 
 \begin{remark} 
Owing to  $\nabla u\in L^1(\R_+;C_b(\Omega)),$   the flow of $u$
has $C^1$ regularity with respect to the space variable, which 
entails the conservation of  the geometrical structures of the fluid during the evolution. For example,  if $\rho_0$ takes  two different positive values across a $C^1$ interface, 
then it remains so  forever: the interface is just transported by the flow and keeps its $C^1$ regularity. 
Likewise, the (local) $H^2$ regularity of the interfaces is preserved since $\nabla^2 u\in L^1(\R_+;L^2(\Omega)).$ 
\end{remark} 
\begin{remark}  As said before,  for $\Omega=\R^3$ a result in the same spirit has been obtained 
by P. Zhang in \cite{Zhang19} in the small velocity case (see also \cite{DW}).  An important difference with our situation 
is that in dimension three, the critical   space  for the velocity is 
$\dot B^{1/2}_{2,1}:=[L^2,\dot H^1]_{1/2,1}.$ Hence, it is enough to prove time weighted energy estimates in $L^2$
and $\dot H^1,$ and the relevant critical space for $u_0$ does not depend on $\rho_0.$ 
\end{remark}

To simplify the presentation, 
  we assume in the rest of the paper that $s=1/2.$ We use 
   the short notation $\wt B^0_{\rho_0,1}$ for $\wt B^{0,1/2}_{\rho_0,1}.$
   
   Let us briefly present  the main steps of the proof of Theorem \ref{thm:2}.  
The global existence of a solution being ensured by prior results, the main point is
to exhibit enough regularity of the solution  to ensure uniqueness. 
As already explained at length in the introduction, the key is to establish \eqref{eq:Lip}, and this 
will be actually performed on the linear system \eqref{eq:LINS0}.
\smallbreak
The first step is to prove energy type weighted estimates for \eqref{eq:LINS0}  that require only $u_0$ to be in $L^2$
and the density to be bounded and bounded away from zero. 
The three principles guiding our search for estimates are: 
\begin{itemize}
\item   taking \emph{convective derivatives}   $D_t:=\partial_t+v\cdot\nabla$ (since $D_t\rho=0$)
rather than  space derivatives   since $\rho$ has no regularity;
\item  using  differential  operators $\sqrt t\nabla,$ $t\partial_t$  and $tD_t$ 
 (that are  of order~$0$  in the parabolic scaling);
 \item transferring  time regularity to space regularity by means of the    maximal regularity 
  properties of the Stokes system (see the Appendix), observing~that
\begin{equation}\label{eq:stokes}
\mu\Delta u-\nabla P=\rho\dot u\andf \div u=0\quad\hbox{in }\ \Omega, \with \dot u:=\d_tu+v\cdot\nabla u.
\end{equation} 
 \end{itemize}
 In the end, this allows to control quantities like    $\|\sqrt t\nabla u(t)\|_{L^2},$   $\|t \partial_t u(t)\|_{L^2},$ 
 $\|t \dot  u(t)\|_{L^2}$  or $\|t\nabla^2 u(t)\|_{L^2}$ 
 in terms of $\|u_0\|_{L^2},$ $\rho_*,$ $\rho^*$ and energy-like norms of $v$. 
\smallbreak 
The second step  is to propagate the $\dot H^1$ and the $\dot H^{-1}$ norms. 
On the one hand,  $\dot H^1$ estimates for (INS)  are known  since  the work by O. Ladyzhenskaya and V. A Solonnikov
in \cite{LS} (we shall also derive  time weighted versions of these estimates). 
On the other hand,  propagating \emph{negative} Sobolev  regularity seems to be new. 
This  will be achieved  by duality after observing that the backward system associated with \eqref{eq:LINS0} satisfies the same
family of estimates in $\dot H^s.$  However, owing the to density dependent structure of the latter system, 
we will have only access to  $\|\cP(\rho u)(t)\|_{\dot H^{-s}},$ whence  the `weighted' definition 
of the interpolation space $\wt B^{0,s}_{\rho,1}.$ 
\smallbreak
The third step is devoted to propagating the regularity  $\wt B^0_{\rho,1}$ 
and to bounding  $\nabla u$ in $L^1(\R_+;L^\infty)$ in terms of the data only. 
In passing, we  exhibit some controls of other critical norms (like e.g. that of $\dot u$ in $L^1(\R_+;L^2)$)
that will be needed in  the proof of uniqueness and  stability. 
All these bounds rely on the dynamic interpolation method that has been described above for the heat equation. 
 In the end,  we get:
 $$ \int_{0}^\infty\|\nabla u\|_{L^\infty}\,dt+\int_0^\infty\|\dot u\|_{L^2}\,dt
 +\biggl(\int_0^\infty t^{2/3}\|\dot u\|_{L^4}^{4/3}dt\biggr)^{3/4} \leq C\|u_0\|_{\wt B^0_{\rho_0,1}}.$$
The fourth step is the proof of existence of a global solution corresponding to the assumptions of Theorems 
\ref{thm:1} or \ref{thm:2}. For Theorem \ref{thm:1}, the overall strategy is standard: we smooth out the data, 
resort to  classical results that ensure the existence of a sequence of global smooth solutions for (INS), 
and use the aforementioned estimates and compactness  to pass to the limit. 
For Theorem \ref{thm:2}, it is a bit the same, except that one has to be careful when smoothing out the velocity,
owing to the `exotic' definition of the space $\wt B^0_{\rho_0,1}.$
The easiest way is to truncate  a decomposition of $u_0$
so as to have an approximate initial velocity in the smoother space $H^{1/2}.$ 
\smallbreak
The last step is devoted to uniqueness and stability for (INS). As in \cite{DM1}, we reformulate (INS) in Lagrangian coordinates. 
The properties of the solutions provided by Theorem \ref{thm:2}, in particular \eqref{eq:Lip},
 ensure that  the two formulations are equivalent. The gain is that we do not have to worry about the density as it is 
 time-independent. As for the difference of the two velocities in Lagrangian coordinates, it  satisfies a parabolic type equation
 and  may be estimated in $L^\infty(\R_+;L^2) \cap L^2(\R_+; \dot H^1).$ 
 The computations are in the spirit of those of  \cite{DMP}. However, in our case the velocity is less regular by one derivative, which requires some care. 
 \smallbreak
 As a concluding remark, we want to point out that, in contrast with numerous recent works dedicated 
 to the inhomogeneous incompressible Navier-Stokes equations,  our approach  does not 
 use Fourier analysis at all. It just relies on very basic  energy arguments, interpolation, embedding and on the classical regularity theory for the Stokes system (this is the only place where some assumptions have to be made on the fluid domain). 
 For simplicity here we  considered  $\R^2,$ $\T^2$ or $C^2$ bounded domains, but   more 
 general domains could be treated in the same way.  
  \medbreak
 In the rest of the paper, we  shall focus on the case $\mu=1$ for simplicity.  
 The general case follows thanks to the  rescaling:
 $$\rho(t,x):= \wt\rho(\mu t,x),\quad  u(t,x):=\mu\wt u(\mu t,x),\quad P(t,x):= \mu^2\wt P(\mu t,x).$$

 
\section{Weak solutions with time decay}   	\label{s:weak}

This section is devoted to proving Theorem \ref{thm:1}:
we here construct finite energy weak solutions satisfying algebraic 
time decay estimates of different orders, 
without requiring more regularity on $u_0$ than $L^2.$
The  exponential decay  that can be expected in the bounded domain case (see  \cite{DMP}),  is not addressed  to simplify the presentation, 
as it is  not needed for achieving the main result of the paper. 

 \subsection{Time decay estimates for the linearized momentum equation}
 
 We here aim at proving time weighted energy estimates for the linear system \eqref{eq:LINS0}
 in the case where the 
(smooth enough) given  pair $(\rho,v)$ satisfies \eqref{eq:condrhov} and 
\begin{equation}\label{eq:rho0}\rho_*=\underset{(t,x)\in\R_+\times\Omega}{\inf}\rho(t,x)>0\andf 
\rho^*=\underset{(t,x)\in\R_+\times\Omega}{\sup}\rho(t,x)<\infty.\end{equation}
System \eqref{eq:LINS0} is supplemented with 
a divergence free initial velocity field $u_0,$ vanishing at the boundary in the bounded domain case 
and, in the torus case, such that $$\int_{\T^2}(\rho_0 u_0)(x)\,dx=0. $$
This latter assumption is  not restrictive owing to the Galilean invariance of the system, and will enable us to use freely the Gagliardo-Nirenberg inequality \eqref{eq:GNT}. 
\medbreak
We aim at  proving   energy  estimates for the solution
with time weights $t^{k/2}$ for $k\in\{0,1,2,3\}$. We strive for bounds  depending only on
$\rho_*,$ $\rho^*,$   $\|u_0\|_{L^2}$ and on  \emph{energy-type norms of $v$} in the meaning given 
at the end of the introduction of the paper. This latter point  is 
 fundamental for getting not only  Theorem \ref{thm:1} but also Theorem \ref{thm:2}.  
 \medbreak
 Before proceeding, let us  warn the reader that we unfortunately did not find a way 
 to avoid the tedious calculations that will follow, since it is has to be checked with the greatest care that only 
 `energy type norms' come into play. 

\subsubsection{The basic energy balance}

Taking the $L^2$ scalar product of \eqref{eq:LINS0} with $u$  yields
\begin{equation}\label{eq:L2dt}
\frac12\frac d{dt}\|\sqrt\rho\, u\|_{L^2}^2+ \|\nabla u\|_{L^2}^2 = 0.
\end{equation}
From this, 
 we get for all $t\in\R_+,$ 
\begin{equation}\label{eq:L2}
\|(\sqrt \rho\,u)(t)\|_{L^2}^2 +
2\int_0^t\|\nabla u\|_{L^2}^2\,d\tau=
\|\sqrt \rho_0\,u_0\|_{L^2}^2.
\end{equation}
As $\rho_*>0,$  combining  \eqref{eq:L2} 
with the Gagliardo-Nirenberg inequality \eqref{eq:GN} recalled in Appendix 
yields for all $2\leq p<\infty$:
\begin{equation}\label{eq:L2ter}
\|u\|_{L^q(L^p)}\leq C_p\,\rho_*^{-1/2} \|\sqrt{\rho_0}\, u_0\|_{L^2}\with 1/p+1/q=1/2. 
\end{equation}

\subsubsection{Estimates with weight $\sqrt t$}  
Let us  rewrite   \eqref{eq:LINS0} as follows: 
 \begin{equation}\label{edu1}
\Delta u - \nabla P=  \rho \t{u}\andf \div u=0\quad\hbox{in }\ \Omega,\ \with \dot u:=u_t+v\cdot\nabla u.
\end{equation}
Taking the $L^2(\Omega;\R^2)$ scalar product of \eqref{edu1} with $t\t{u}$ yields  for all $t\geq0$:
$$\int_{\Omega}\rho t|\t{u}|^2\,dx=t\int_{\Omega}\Delta u\cdot u_t\,dx
-t\int_{\Omega}\nabla P\cdot u_t\,dx + 
t\int_{\Omega}\bigl(\Delta u-\nabla P)\cdot(v\cdot\nabla u)\,dx.$$
As $\div u=0,$ integrating by parts and using again \eqref{edu1} yields
\begin{equation}\label{eq:weight0}
\frac{1}{2}\frac{d}{dt}\int_{\Omega}t\abs{\nabla u}^{2}dx
-\frac12\int_{\Omega}\abs{\nabla u}^2\,dx
+\int_{\Omega} \rho t\abs{\t{u}}^{2}\,dx=\int_{\Omega}\rho t\t{u}\cdot(v\cdot \nabla u)\,dx.\end{equation}
Remembering   \eqref{eq:L2dt} and performing a time integration, we get for all $t\geq 0,$
\begin{multline}\label{eq:weight1}
\frac14\int_{\Omega} \rho(t)|u(t)|^2\,dx+
\frac{t}{2}\int_{\Omega}\!\abs{\nabla u(t)}^{2}dx+\int_{0}^{t}\!\!\int_{\Omega} \tau \rho \abs{\t{u}}^{2}dx\,d\tau\\
=\frac14\int_{\Omega} \rho_0|u_0|^2\,dx+ \int_{0}^{t}\!\!\int_{\Omega}\!\! \tau \rho \t{u}\cdot(v\cdot \nabla u)\,dx\, d\tau.\end{multline}
Of course, since  $u_t=\dot u-v\cdot\nabla u,$ one can write
$$\frac14\|\sqrt\rho u_t\|_{L^2}^2\leq \frac12 \|\sqrt\rho\, \dot u\|_{L^2}^2+\frac12\|\sqrt\rho\,v\cdot\nabla u\|_{L^2}^2.$$
Hence adding up this inequality multiplied by $t,$ to \eqref{eq:weight1} and using Young inequality to bound the last
term of \eqref{eq:weight1}, we discover that 
\begin{multline}\label{eq:weight1b} \|\sqrt{\rho(t)} u(t)\|_{L^2}^2+ 2\|\sqrt t \nabla u(t)\|_{L^2}^2
+\int_0^t\bigl(\|\sqrt{\rho\tau}\,\t{u}\|_{L^2}^2+\|\sqrt{\rho\tau}\,u_\tau\|_{L^2}^2\bigr)d\tau\\
\leq \|\sqrt{\rho_0}u_0\|_{L^2}^2+6\int_0^t\|\sqrt{\rho\tau} v\cdot\nabla u\|_{L^2}^2\,d\tau.\end{multline}
Combining H\"older, Ladyzhenskaya inequality \eqref{eq:lad} and Young inequality yields
\begin{equation}\label{eq:weight1a}
\|\sqrt\rho\, v\cdot\nabla u\|_{L^2}^2\leq  \frac\eps{\rho^*}\|\nabla^2u\|_{L^2}^2
+\frac{\rho^*}{\eps}\|\sqrt\rho v\|_{L^4}^4\|\nabla u\|_{L^2}^2,\qquad\eps>0,\end{equation}
and taking advantage of  the  regularity  theory of the Stokes system  (recalled in Appendix) gives
\begin{equation}\label{ed2}\norm{\nabla^{2}u}_{L^{2}}^2+\norm{\nabla P}_{L^{2}}^2\leq C_\Omega\rho^*
\norm{\sqrt{\rho}\t{u}}_{L^{2}}^2.\end{equation}
Hence, choosing $\eps>0$ suitably small  in \eqref{eq:weight1a}, using  \eqref{ed2}, then reverting 
to \eqref{eq:weight1b} and applying Gronwall lemma allows to conclude that there exist  
positive constants $c_\Omega$ and $C_\Omega$ depending only on $\Omega,$  such that 
\begin{multline}\label{eq:weight3}
X_{1}(t)\leq \|\sqrt \rho_0u_0\|_{L^2}^2  e^{C_1^v(t)}\with 
  C_1^v(t):=C_\Omega\rho^*\!\!\int_0^t\!\|\sqrt\rho\,v\|_{L^4}^4\,d\tau\andf\\
X_{1}(t):= \|(\sqrt{\rho}\,u)(t)\|_{L^2}^2+ 2\|\sqrt t \nabla u(t)\|_{L^2}^2\hspace{6cm}
\\+\frac12\int_0^t \Bigl(\|\sqrt{\rho\tau}\,\t{u}\|_{L^2}^2+\|\sqrt{\rho\tau}\,{u_\tau}\|_{L^2}^2+
\frac{c_\Omega}{\rho^*}\|\sqrt\tau(\nabla^2u,\nabla P)\|_{L^2}^2\!\Bigr)d\tau.
\end{multline}

\subsubsection{Estimates with weight $t$}  Applying $\d_{t}$ to \eqref{eq:LINS0}  gives 
\begin{equation}\label{eq:utt}
    \rho u_{tt}+\rho v\cdot \nabla u_{t}-\Delta u_{t}+\nabla P_{t}=-\rho_{t}\t{u}-\rho v_{t}\cdot \nabla u.
\end{equation}
As $\div u_t=0,$ testing  \eqref{eq:utt} by  $t^{2}u_t$ then observing that 
$$\rho_t=-\div(\rho v)\andf |u_t|^2= |\dot u|^2 -2 \dot u\cdot(v\cdot\nabla u) +|v\cdot\nabla u|^2$$
gives after performing a few integration by parts: 
$$\displaylines{
\quad
\frac{1}{2}\frac{d}{dt}\int_{\Omega}\rho t^{2}\abs{u_{t}}^{2}\,dx+\int_{\Omega}t^{2}\abs{\nabla u_{t}}^{2}\,dx
= \int_{\Omega} t\rho\abs{\dot u}^{2}\,dx-2\int_{\Omega}\rho t \dot u\cdot(v\cdot\nabla u)\,dx\hfill\cr\hfill
+\int_{\Omega} t\rho \abs{v\cdot\nabla u}^2\,dx
+\int_{\Omega}t^{2} \div(\rho v)\t{u}\cdot u_{t}\,dx-\int_{\Omega}t^{2}\rho (v_{t}\cdot \nabla u)\cdot u_{t}\,dx.\quad}$$
Adding up  twice  \eqref{eq:L2dt}  and  \eqref{eq:weight0} to this latter inequality, we obtain:
\begin{multline}\label{eq:II2}
\frac{d}{dt}\int_{\Omega}\Bigl(\rho|u|^2+ t\abs{\nabla u}^{2}+\frac{\rho t^{2}}2\abs{u_{t}}^{2}\Bigr)dx
+\int_{\Omega} \bigl(\abs{\nabla u}^2+  \rho t\abs{\t{u}}^{2}+t^2\abs{\nabla u_t}^2\bigr)dx\\
=\int_{\Omega}\rho t\abs{v\cdot \nabla u}^2\,dx
+\int_{\Omega}t^{2} \div(\rho v)\,\t{u}\cdot u_{t}\,dx
-\int_{\Omega}t^{2}\rho (v_{t}\cdot \nabla u)\cdot u_{t}\,dx=:I_{1}+I_{2}+I_{3}.\end{multline}
Thanks to \eqref{eq:weight1a}, \eqref{ed2} and  Young inequality, we have
 \begin{equation}\label{eq:I1}
I_1\leq\frac12  \|\sqrt{\rho{t}}\dot u\|_{L^2}^2 + {C\rho^*}\|\sqrt\rho v\|_{L^4}^4\|\sqrt{t}\nabla u\|_{L^2}^2.
\end{equation}
For term $I_{2},$ an integration by parts yields 
$$
    I_{2}
    =-\int_{\Omega}{t}^{2} (\rho v\cdot \nabla{\t{u}}) \cdot u_{{t}}\,dx
    -\int_{\Omega}{t}^{2} (\rho v\cdot \nabla{u_{{t}})\cdot  \t{u} }\,dx=:I_{21}+I_{22}.$$
    By \eqref{eq:lad}, H\"older and Young inequalities, and \eqref{eq:rho0}, we have for some constant 
    $C$ depending only on $\rho_*,$ $\rho^*$ and $\Omega,$
     \begin{align}\label{eq:I21}
    I_{21}    &\leq C  \|t\nabla\dot u\|_{L^2}\|\sqrt \rho v\|_{L^4}\|tu_t\|^{1/2}_{L^2}\|t\nabla u_t\|^{1/2}_{L^2}\nonumber\\
    &\leq \frac1{10}\bigl(\|t\nabla u_t\|_{L^2}^2+\|t\nabla\dot u\|_{L^2}^2\bigr)
    +C\|\sqrt \rho v\|_{L^4}^4\|\sqrt\rho\,  tu_t\|_{L^2}^2.   \end{align}
 The same arguments lead to 
 \begin{equation}\label{eq:I22}
 I_{22}\leq  \frac1{10}\bigl(\|t\nabla u_t\|_{L^2}^2+\|t\nabla\dot u\|_{L^2}^2\bigr)
    +C\|\sqrt \rho v\|_{L^4}^4\|\sqrt\rho\, t \dot u\|_{L^2}^2.  \end{equation}
For $I_{3}$, one has, still owing to H\"older and Young inequalities,  and \eqref{eq:GN} or \eqref{eq:GNT}, 
\begin{align}\label{eq:I3} I_{3}&\leq \norm{\sqrt{\rho{t}}\, v_{{t}}}_{L^{2}} \norm{{t}\sqrt{\rho} \, u_{{t}}}_{L^{4}}\norm{\sqrt{t}\nabla u}_{L^{4}}\nonumber\\ 
    &\leq  \frac1{10}\|t\nabla u_t\|_{L^2}\|\nabla u\|_{L^2}
+C\norm{\sqrt{\rho{t}}\, v_{{t}}}_{L^{2}}^2  \norm{{t} \sqrt\rho u_{{t}}}_{L^{2}}\|t\nabla^2u\|_{L^2}.        \end{align}
Hence, inserting \eqref{eq:I1}, \eqref{eq:I21}, \eqref{eq:I22} and \eqref{eq:I3} in \eqref{eq:II2} 
gives
\begin{multline}\label{eq:II2a}
\frac{d}{dt}\Bigl(\|\sqrt \rho\,u\|_{L^2}^2+ \|\sqrt t \nabla u\|_{L^2}^{2}+\frac12\|\sqrt\rho t u_t\|_{L^2}^2\Bigr)\\
+\frac12\Bigl(\|\nabla u\|_{L^2}^2+  \|\sqrt{\rho t}\t{u}\|_{L^2}^{2}+\|t\nabla u_t\|_{L^2}^2\Bigr)-\frac14\|t\nabla \dot u\|_{L^2}^2
\\\lesssim\|\sqrt \rho v\|_{L^4}^4\bigl(\|\sqrt\rho t (\dot u,u_t)\|_{L^2}^2+
\|\sqrt t \nabla u\|_{L^2}^2\bigr)+\norm{\sqrt{\rho{t}}\, v_{{t}}}_{L^{2}}^2  \norm{{t} \sqrt \rho u_{{t}}}_{L^{2}}\|t\nabla^2u\|_{L^2}.
\end{multline}
To close the estimate, we  have to bound  $\|\sqrt{\rho} t\dot u\|_{L^2},$ $\|t\nabla^2u\|_{L^2}$
and $\|t\nabla\dot u\|_{L^2}.$  For the first two terms, one may use \eqref{eq:lad}, 
 \eqref{ed2} and  the definition of $\dot u$ to get
$$ \begin{aligned}\|t(\nabla ^2 u,\nabla P)\|_{L^2}
  &\leq C_\Omega\bigl(\sqrt{\rho^*}\|t\sqrt\rho u_t\|_{L^2} +\|\rho\, t^{1/4} v\|_{L^4} \|\sqrt t\nabla u\|_{L^2}^{1/2}\|t\nabla^2u\|_{L^2}^{1/2}\bigr)\\
  &\leq \frac12\|t\nabla^2u\|_{L^2}+
  C_\Omega\bigl(\sqrt{\rho^*}\|t\sqrt\rho u_t\|_{L^2} +\|\rho\, t^{1/4} v\|_{L^4}^2 \|\sqrt t\nabla u\|_{L^2}\bigr) \cdotp\end{aligned}$$
 This, in the end, implies that
  \begin{equation}\label{eq:stokest1}
  \frac14\|\sqrt \rho t\dot u\|_{L^2}+\frac{c_\Omega}{\sqrt{\rho^*}} \|t\nabla^2u,t\nabla P\|_{L^2}
  \leq C\bigl(\|t\sqrt\rho u_t\|_{L^2} +\| t^{1/4} v\|_{L^4}^2 \|\sqrt t\nabla u\|_{L^2}\bigr) \cdotp\end{equation}
        Finally, from the definition of $\dot u,$ H\"older inequality and \eqref{eq:lad}, we may write:    
        $$\begin{aligned}
        \|t\nabla\dot u\|_{L^2}&\leq \|t\nabla u_t\|_{L^2}
        +\|t\nabla v\cdot\nabla u\|_{L^2}+\|tv\cdot\nabla^2 u\|_{L^2}\\
        &\leq   \|t\nabla u_t\|_{L^2}+\|\sqrt t\nabla v\|_{L^4}\|\nabla u\|_{L^2}^{1/2}\|t\nabla^2u\|_{L^2}^{1/2}+
        C \|v\|_{L^4} \|t\dot u\|_{L^2}^{1/2}\|t\nabla\dot u\|_{L^2}^{1/2},
\end{aligned}        $$
which implies that         \begin{equation}\label{eq:tdudot}
             \|t\nabla\dot u\|_{L^2}\leq 2 \|t\nabla u_t\|_{L^2}+
             \frac{\|\nabla u\|_{L^2}}4
             +C\bigl(\|\sqrt t\nabla v\|_{L^4}^2\|t\nabla^2 u\|_{L^2}\!+\!\|v\|_{L^4}^2\|\sqrt\rho t\dot u\|_{L^2}\bigr)\cdotp
             \end{equation}
     Let us set
             $$\displaylines{
  X_2(t):=\|(\sqrt\rho u)(t)\|_{L^2}^2+\|\sqrt t\nabla u(t)\|_{L^2}^2+\frac14\|\sqrt\rho tu_t\|_{L^2}^2
  +\frac1{16}\|\sqrt\rho t\dot u\|_{L^2}^2\hfill\cr\hfill+\frac{c_\Omega}{\rho^*}\|t(\nabla^2u,\nabla P)\|_{L^2}^2
  +\frac1{16}\int_0^t\bigl(\|\nabla u\|_{L^2}^2+\|\sqrt{\rho\tau}\dot u\|_{L^2}^2+
    \|\tau\nabla u_\tau\|_{L^2}^2 + \|\tau\nabla\dot u\|_{L^2}^2\bigr)d\tau.             }$$
             Integrating \eqref{eq:II2a} on $[0,t],$ then taking advantage of 
             \eqref{eq:stokest1} and \eqref{eq:tdudot}, then, finally, 
             using Gronwall lemma, we conclude that there exists a constant $C$ depending only on 
             $\Omega,$ $\rho_*$ and $\rho^*$ such that
      \begin{multline}\label{eq:weighttt}
      X_2(t)\leq \|u_0\|_{L^2}^2e^{C_2^v(t)}\with\\
 C_2^v(t):= C\biggl(\underset{\tau\in[0,t]}{\sup}\|\tau^{1/4} v(\tau)\|_{L^4}^4+\int_0^t \bigl(\|\sqrt\rho\,v\|_{L^4}^4+\|\sqrt\tau\nabla v\|_{L^4}^4+\|\sqrt{\rho\tau}v_\tau\|_{L^2}^2\bigr)d\tau\biggr)\cdotp\end{multline}

\subsubsection{Estimates with weight $t^{3/2}$}
Let  $D_t:=\d_t+v\cdot\nabla$ and $\ddot u:=D_t\dot u.$   We have\footnote{Here 
we use the notation $(\nabla^2 u \cdot \nabla v)^{i}:=\underset{1\leq j,k\leq d}{\Sum} \d_{k}v^{j} \,\d_{j}\d_{k} u^{i}.$}: 
\begin{equation}\label{eq:ddotu}
\rho \ddot u-\Delta\dot u+\nabla\dot P= F
:=\nabla v\cdot\nabla P-\Delta v\cdot\nabla u -2\nabla^2 u\cdot\nabla v.
\end{equation}
 Taking the  $L^2(\Omega;\R^2)$ scalar product with $t^3\ddot u,$  we readily get 
\begin{equation}\label{eq:weightIV-1}
\frac12\frac d{dt}\|t^{3/2}\nabla\dot u(t)\|_{L^2}^2+\|t^{3/2}\sqrt\rho\,\ddot u\|_{L^2}^2
=\frac32\|t\nabla \dot u\|_{L^2}^2+\sum_{i=1}^5 J_i\end{equation}
with 
$$\begin{aligned}
J_1&:= \int_{\Omega}\Delta \dot u\cdot({t}^3 v\cdot\nabla \dot u)\,dx,\\
J_2&:=- \int_{\Omega}\nabla\dot P\cdot\bigl({t}^3 v\cdot(\nabla v\cdot \nabla  u)\bigr)dx,\\
J_3&:=\int_{\Omega}  \nabla\dot P\cdot({t}^3v_{t}\cdot\nabla u) \,dx,\\
J_4&:= \int_{\Omega}\nabla\dot P\cdot\bigl({t}^3v\cdot(v\cdot\nabla^2 u)\bigr)dx,\\
J_5&:=\int_{\Omega} F\cdot{t}^3\ddot u \,dx.
\end{aligned}$$
For any $\eps>0,$ 
the terms $J_1$ to $J_5$ may be bounded as follows by combining H\"older inequality, Young inequality,
\eqref{eq:GN} with $p=4$ or $p=6$ (and \eqref{eq:stokesLp} for $J_4$): 
$$
J_1\leq \|{t}^{3/2}\nabla^2\dot u\|_{L^2}\|v\|_{L^4}\|{t}^{3/2}\nabla\dot u\|_{L^4}
\leq \eps\|{t}^{3/2}\nabla^2\dot u\|_{L^2}^2+C_\eps\|v\|_{L^4}^4\|{t}^{3/2}\nabla\dot u\|_{L^2}^2,$$
$$\begin{aligned} 
J_2&\leq \|{t}^{3/2}\nabla \dot P\|_{L^2}\|{t}^{1/6}v\|_{L^6}\|\sqrt {t}\nabla v\|_{L^6}
\|{t}^{5/6}\nabla  u\|_{L^6}\\
&\leq C \|{t}^{3/2}\nabla \dot P\|_{L^2}\|{t}^{1/6}v\|_{L^6}\|\sqrt {t}\nabla v\|_{L^6}
\|\sqrt{t}\nabla  u\|_{L^2}^{1/3} \|{t} \nabla^2 u\|_{L^6}^{2/3}\\
&\leq \eps\|{t}^{3/2}\nabla\dot P\|_{L^2}^2
+C_\eps\|{t}^{1/6}v\|_{L^6}^2\|\sqrt {t}\nabla v\|_{L^6}^2\|\sqrt{t}\nabla  u\|_{L^2}^{2/3} \|{t} \nabla^2 u\|_{L^2}^{4/3}  ,
\end{aligned}$$
$$\begin{aligned} 
J_3&\leq\|{t}^{3/2}\nabla\dot P\|_{L^2}\|{t} v_{t}\|_{L^4}\|{t}^{1/2}\nabla u\|_{L^4}\\
&\leq \eps\|{t}^{3/2}\nabla\dot P\|_{L^2}^2+C_\eps\|{t} v_{t}\|_{L^4}^4
\|{t}^{1/2}\nabla u\|_{L^2}^2+\|{t}^{1/2}\nabla^2u\|_{L^2}^2,
\end{aligned}$$
$$\begin{aligned} 
J_4&\leq \|{t}^{3/2}\nabla\dot P\|_{L^2}\|{t}^{1/6}v\|_{L^6}^2\|{t}^{7/6}\nabla^2 u\|_{L^6}\\
&\leq C \|{t}^{3/2}\nabla\dot P\|_{L^2}\|{t}^{1/6}v\|_{L^6}^2\|\sqrt{\rho{t}}\dot u\|_{L^2}^{1/3}\|{t}^{3/2}\nabla\dot u\|_{L^2}^{2/3}\\
&\leq\eps\|{t}^{3/2}\nabla\dot P\|_{L^2}^2 +C_\eps\|\sqrt{\rho{t}}\dot u\|_{L^2}^{2}
+C_\eps\|{t}^{1/6}v\|_{L^6}^6 \|{t}^{3/2}\nabla\dot u\|_{L^2}^{2},\\
J_5&\leq \eps\|{t}^{3/2}\sqrt\rho\ddot u\|_{L^2}^2 +\frac{C_\eps}{\rho^*} \|{t}^{3/2}F\|_{L^2}^2.\end{aligned}$$
Thanks to H\"older inequality,  \eqref{eq:lad} and  \eqref{eq:stokesLp},  we have 
$$\begin{aligned} 
 \|{t}^{3/2}F\|_{L^2}^2&\leq \|\sqrt t\nabla v\|_{L^4}^2
\|{t}(\nabla P,\nabla^2u)\|_{L^4}^2+\|{t}\nabla^2 v\|_{L^4}^2\|\sqrt t\nabla u\|_{L^4}^2,\\
&\lesssim \|\sqrt t\nabla v\|_{L^4}^2\|\sqrt{\rho{t}}\dot u\|_{L^2}\|{t}^{3/2}\nabla\dot u\|_{L^2}
+\|{t}\nabla^2 v\|_{L^4}^2\|\sqrt t\nabla u\|_{L^2}\|\sqrt t\nabla^2u\|_{L^2}\\
 &\lesssim \|\sqrt{\rho{t}}\dot u\|_{L^2}^2\!+\!\|\sqrt t\nabla^2 u\|_{L^2}^2\!+\!\|\sqrt t\nabla v\|_{L^4}^4\|{t}^{3/2}\nabla\dot u\|_{L^2}^2
 \!+\!\|{t}\nabla^2v\|_{L^4}^4\|\sqrt t\nabla u\|_{L^2}^2. 
\end{aligned}$$
To close the estimates, we need to bound $t^{3/2}\nabla\dot P$ and $t^{3/2}\nabla^2\dot u$  in 
$L^2(\R_+\times\Omega).$
Now, we observe that the couple $(\dot u,\nabla\dot P)$ satisfies the inhomogeneous Stokes system
\begin{equation}\label{eq:inhomo}-\Delta \dot u+\nabla\dot P= F-\rho\ddot u\andf \div\dot u= {\rm Tr}(\nabla v\cdot\nabla u)\quad
\hbox{in }\ \Omega\end{equation}
with boundary condition $\dot u|_{\d\Omega}=0$ if $\Omega$ is a bounded domain, $\dot u(t)\to0$ at infinity
(due to $\dot u(t)\in L^2$ for all $t>0$) in the case $\Omega=\R^2,$ and 
$$\int_{\T^2}\rho\dot u\,dx=0\quad\hbox{if}\quad \Omega=\T^2.$$
Hence, applying \eqref{eq:stokesLp} with $p=2$ guarantees that 
\begin{equation}\label{eq:stokes3/2}\|\nabla^2\dot u,\nabla\dot P\|_{L^2}^2\lesssim \|F\|_{L^2}^2
+\|\rho\ddot u\|_{L^2}^2 +\|\nabla^2 v\otimes\nabla u\|_{L^2}^2+\|\nabla v\otimes\nabla^2 u\|_{L^2}^2.\end{equation}
The last two terms are parts of $F.$ Hence bounding $\|{t}^{3/2}F\|_{L^2}$ as above
and putting together  with  the previous inequalities, we conclude after time integration that
$$
\displaylines{X_{3}(t)
:=\|t^{3/2}\nabla\dot u(t)\|_{L^2}^2+\int_0^t\|\tau^{3/2}(\sqrt\rho\,\ddot u,\nabla\dot P,\nabla^2\dot u)\|_{L^2}^2\,d\tau
\hfill\cr\hfill\lesssim\int_0^t\bigl(\|v\|_{L^4}^4\!+\!\|\tau^{1/6}v\|_{L^6}^6\!+\!\|\tau^{1/2}\nabla v\|_{L^4}^4\bigr)\|\tau^{3/2}\nabla\dot u\|_{L^2}^2\,d\tau
+\int_0^t\|\tau^{1/2}\nabla^2u,\sqrt{\rho\tau}\dot u\|_{L^2}^2\,d\tau
\hfill\cr\hfill+\int_0^t\bigl(\|\tau v_\tau\|_{L^4}^4+\|\tau\nabla^2v\|_{L^4}^4\bigr)\|\tau^{1/2}\nabla u\|_{L^2}^2\,d\tau
\hfill\cr\hfill+\int_0^t\|\tau^{1/6}v\|_{L^6}^2\|\sqrt\tau\nabla v\|_{L^6}^2\|\sqrt\tau\nabla u\|_{L^2}^{2/3}
\|\tau\nabla^2 u\|_{L^2}^{4/3}\,d\tau.}
$$
After using Gronwall lemma and the inequalities of the previous steps, we get
\begin{multline}\label{eq:weight3/2}
X_{3}(t)\leq  C\|u_0\|_{L^2}^2  e^{C_{3}^v(t)}\with\\
C_{3}^v(t):=C\int_0^t\bigl(\|v\|_{L^4}^4+(1+\|\tau^{1/4}v\|_{L^4}^4)\|v\|_{L^6}^3+\|\tau^{1/6}v\|_{L^6}^6
+\|\sqrt\tau\nabla v\|_{L^6}^3+\|\tau^{1/2}v_\tau\|_{L^2}^2
\\+\|\tau^{1/2}\nabla v\|_{L^4}^4+\|\tau\nabla^2v\|_{L^4}^4 +\|\tau v_\tau\|_{L^4}^4\bigr)d\tau. 
\end{multline}


\subsection{The proof of Theorem \ref{thm:1}}

Let us fix some data $(\rho_0,u_0)$ such that $u_0\in L^2$ and $0<\rho_*\leq\rho_0\leq\rho^*<\infty.$
Then we smooth out the velocity so as to get a sequence $(u_0^n)_{n\in\N}$  of $H^1$ divergence 
free vector-fields (vanishing at  $\d\Omega$ in the bounded domain case) that converges
strongly to $u_0$ in $L^2.$ 
It is known (see \cite{DM1} for the bounded domain or torus cases, and \cite{PZZ} for the 
$\R^2$ case)  that such data generate a unique global solution $(\rho^n,u^n,\nabla P^n)$ with 
relatively smooth velocity.  In particular, the computations leading to the estimates of
the previous subsection may be justified for $\rho=\rho^n,$ $u=v=u^n$,  and we get for all $t\geq0$
for some constant depending only on $\rho_*,$ $\rho^*$ and $\Omega,$
\begin{equation}\label{eq:L2n}
X_0^n(t):= \|(\sqrt{\rho^n}\,u^n)(t)\|_{L^2}^2 +2\int_0^t\|\nabla u^n\|_{L^2}^2\,d\tau\leq \|\sqrt \rho_0\,u_0^n\|_{L^2}^2,
\end{equation}
\begin{equation}\label{eq:weight3n}
X^n_{1}(t)
\leq \|\sqrt \rho_0u_0^n\|_{L^2}^2  e^{C_1^n(t)} \with C_1^n(t):= C\int_0^t \|u^n\|_{L^4}^4\,d\tau,
\end{equation} 
  \begin{multline}\label{eq:weightttn}
      X^n_2(t)\leq \|\sqrt\rho_0\,u_0^n\|_{L^2}^2  e^{C_2^n(t)}
      \with \\ C_2^n(t):= C\Bigl(\underset{\tau\in[0,t]}{\sup}\|\tau^{1/4} u^n(\tau)\|_{L^4}^4+
 \int_0^t \bigl(\|u^n\|_{L^4}^4+\|\sqrt\tau\nabla u^n\|_{L^4}^4+\|\sqrt{\tau}u^n_\tau\|_{L^2}^2\bigr)d\tau\Bigr),\end{multline}
 \begin{multline}\label{eq:weight3/2n}
X^n_{3}(t)\leq  C\|u_0^n\|_{L^2}^2   e^{C_3^n(t)}\with 
C^n_3(t):=C\int_0^t\bigl((1+\|\tau^{1/4}u^n\|_{L^4}^4)\|u^n\|_{L^6}^3\\+\|\tau^{1/6}u^n\|_{L^6}^6
+\|\sqrt\tau\nabla v^n\|_{L^6}^3+\|\tau^{1/2}v^n_\tau\|_{L^2}^2
+\|u^n,\tau^{1/2}\nabla u^n,\tau\nabla^2u^n,\tau u^n_\tau\|_{L^4}^4\bigr)d\tau. 
\end{multline}
Above, $X^n_j$ for $j\in\{1,2,3\}$   are the quantities defined in \eqref{eq:weight3}, \eqref{eq:weighttt} and \eqref{eq:weight3/2},
respectively,  pertaining
to $(\rho^n,u^n,\nabla P^n).$ 
\smallbreak
The fundamental point is that all the norms coming into play in $C_1^n,$ $C_2^n$ and $C_3^n$ may be bounded by means of 
  $M:=\sup_{n\in\N} \|u_0^n\|_{L^2},$ $\rho_*$ and $\rho^*.$ 
  For $C_1^n,$ this just stems from \eqref{eq:L2ter} with $p=4.$ 
  Hence we have for some $C_M:= C(\rho_*,\rho^*,M),$
  $$
\underset{t\in\R_+}{\sup} X^n_{1}(t)\leq C_M.$$
Combining with  \eqref{eq:lad} and \eqref{eq:L2n}, we thus get
\begin{align}
\underset{t\in\R_+}{\sup}  \|t^{1/4}u^n(t)\|_{L^4}^4&\lesssim  \|u^n\|_{L^\infty(L^2)}^2\|\sqrt t\nabla u^n\|_{L^\infty(L^2)}^2
\lesssim M^2C_M,\label{eq:est1}\\
\|\sqrt t\nabla u^n\|_{L^4(L^4)}^4 &\lesssim  \|\sqrt t \nabla u^n\|_{L^\infty(L^2)}^2\|\sqrt t\nabla^2 u^n\|_{L^2(L^2)}^2
\lesssim C_M^2,\label{eq:est2}\\
\|\sqrt{\rho t}u^n_t\|_{L^2(L^2)}^2&\lesssim C_M,\label{eq:est3}
\end{align}
whence, remembering \eqref{eq:weightttn}, we have up to a change of $C_M,$   
$$X^n_2(t)\leq C_M\quad\hbox{for all }\ t\geq0.$$ 
Finally, one has to bound the terms of $C^n_3$ independently of $n.$ Let us just treat the third one as an example. 
We write that, owing to \eqref{eq:GN} with $p=6,$ 
$$\begin{aligned}\int_0^\infty\|t^{1/6}u^n\|_{L^6}^6\,dt&\lesssim\int_0^\infty \|u^n\|_{L^2}^2 \|\sqrt t\nabla u^n\|_{L^2}^2
\|\nabla u^n\|_{L^2}^2\,dt\\
&\leq \|u^n\|_{L^\infty(L^2)}^2 
\|\sqrt t\nabla u^n\|_{L^\infty(L^2)}^2
\|\nabla u^n\|_{L^2(L^2)}^2 \lesssim M^4C_M.\end{aligned}$$
As a conclusion, we deduce that there exists a constant, still denoted by $C_M$ such that, for all $n\in\N,$ we have
$$ \underset{t\in\R_+}{\sup} \bigl(X_0^n(t)+X_{1}^n(t)+X_2^n(t)+X_{3}^n(t))\leq  C_M.$$
Regarding the density, the divergence free property of $u^n$ clearly ensures that
$$\forall n\in\N,\,\forall t\in\R_+,\; \rho_*\leq \rho^n(t)\leq \rho^*.$$
At this point, arguing like in the classical proofs of global existence of weak solutions for (INS)  (see e.g. \cite{BF,PLL}), 
one can conclude that 
$(\rho^n,u^n,\nabla P^n)_{n\in\N}$ converges weakly, up to subsequence
to a global distributional solution of (INS) satisfying not only \eqref{eq:rho0} and the usual energy inequality
\eqref{eq:L2INS},  but also 
$$ \underset{t\in\R_+}{\sup} \bigl(X_{1}(t)+X_2(t)+X_{3}(t)\bigr)\leq  C_{\rho_*,\rho^*,\|u_0\|_{L^2}}.$$

 
\section{More decay estimates}   \label{s:lip}

The goal of this section is to prove that the solutions to the linearized momentum equation \eqref{eq:LINS0} 
with $\rho$ satisfying \eqref{eq:rho0} and $v$ verifying the regularity properties listed in Theorem \ref{thm:1}, 
supplemented with divergence free $u_0$ in $\wt B^0_{\rho_0,1}$ satisfy \eqref{eq:Lip}. 
Achieving the result requires several steps. The cornerstones are 
 estimates in $\dot H^1$ and $\dot H^{-1}$ for the solution to \eqref{eq:LINS0} 
 (in addition to the estimates that have been proved hitherto), and the interpolation method that has been described 
 in Section \ref{s:results}.

	\subsection{A priori estimates involving $\dot H^1$ regularity of $u_0$}

In this part, we consider System  \eqref{eq:LINS0}  with some source term $g.$ Our aim is to prove
estimates of $u$ in $\dot H^1,$ in terms of $\nabla u_0\in L^2$ and $g$ in $L^2(L^2).$ 
Considering here  a source term  will be needed when proving
estimates in $\dot H^{-1}$ by means of a duality method.  

\subsubsection{Basic estimates in $\dot H^1$}  Let $f:=g/\rho.$ Taking 
the $L^2$ scalar product of the first line of \eqref{eq:LINS0} with $u_t$ yields, after integrating by parts in the term with $\Delta u,$
\begin{equation}\label{eq:H1}
\frac12\frac d{dt}\|\nabla u\|_{L^2}^2+\|\sqrt\rho\, u_t\|_{L^2}^2= \int_{\Omega} \sqrt\rho(f-v\cdot\nabla u)\cdot (\sqrt\rho\,u_t)\,dx.\end{equation}
By virtue of Young and H\"older inequality, we have 
$$
 \int_{\Omega} \sqrt\rho(f-v\cdot\nabla u)\cdot (\sqrt\rho\,u_t)\,dx\leq \frac12\|\sqrt\rho\, u_t\|_{L^2}^2 
 +\|\sqrt\rho\, f\|_{L^2}^2 + \|\sqrt\rho\, v\cdot\nabla u\|_{L^2}^2.
 $$
Since  $\dot u = u_t +v\cdot\nabla u,$ we may write
$$\|\sqrt\rho\,\dot u\|_{L^2}\leq \|\sqrt\rho\,u_t\|_{L^2} +\|\sqrt\rho\,v\cdot\nabla u\|_{L^2}.$$
Remembering  \eqref{eq:weight1a},  this yields for some constant $c_\Omega$ depending only on $\Omega$:
\begin{equation}\label{eq:H1a}
\frac d{dt}\|\nabla u\|_{L^2}^2+\frac14\|\sqrt\rho\, (u_t, \dot u)\|_{L^2}^2
+\frac{c_\Omega}{\rho^*}\|\nabla^2u,\nabla P\|_{L^2}^2
\leq 4\|\sqrt\rho\, f\|_{L^2}^2.\end{equation}
In the end, combining with Gronwall lemma and remembering that $f=g/\rho,$ we get 
 \begin{multline}\label{eq:H1ter}
\|\nabla u(t)\|_{L^2}^2+ \frac14 \int_0^t\|\sqrt\rho\, (u_t, \dot u)\|_{L^2}^2\,d\tau +\frac{c_\Omega}{\rho^*}
\int_0^t\|\nabla^2 u,\nabla P\|_{L^2}^2\,d\tau
\\\leq e^{C\rho^*\int_0^t\|\sqrt\rho\, v\|_{L^4}^4\,d\tau} \biggl(\|\nabla u_0\|_{L^2}^2
+ 4\int_0^te^{-C\rho^*\int_0^\tau\|\sqrt\rho\, v\|_{L^4}^4\,d\tau'}
\|g/\sqrt\rho\|_{L^2}^2\,d\tau\biggr)\cdotp
\end{multline}

\subsubsection{Decay  estimates  with weight $\sqrt t$:} 

Assuming in the rest of this part that  $g\equiv0$, we proceed as for proving \eqref{eq:weighttt} 
except that we take the $L^2$ scalar product of
\eqref{eq:utt} with $tu_t,$ instead of $t^2 u_t.$ In this way, we get
\begin{multline}\label{eq:H1half}
\frac12\frac d{dt}\Bigl(\|\sqrt{\rho t}u_t\|_{L^2}^2+\frac12\|\nabla u\|_{L^2}^2\Bigr)+\|\sqrt t\nabla u_t\|_{L^2}^2
 \\=\int_\Omega t\div(\rho v) \dot u\cdot u_t\,dx
-\int_\Omega t\rho (v_t\cdot\nabla u)\cdot u_t\,dx-\int_\Omega\rho(v\cdot\nabla u)\cdot u_t\,dx.
\end{multline}
Combining \eqref{eq:GN}, Young inequality  and \eqref{eq:weight1a} gives
$$-2\int_\Omega \rho(v\cdot\nabla u)\cdot u_t\,dx
\leq \frac12\|\sqrt \rho u_t\|_{L^2}^2 +\frac{c_\Omega}{\rho^*}\|\nabla^2u\|_{L^2}^2+C\rho^*\|\sqrt\rho v\|_{L^4}^4\|\nabla u\|_{L^2}^2.
$$
Hence, adding up half \eqref{eq:H1a} to \eqref{eq:H1half} yields
\begin{multline}\label{eq:H1half2}
\frac12\frac d{dt}\Bigl(\|\sqrt{\rho t}u_t\|_{L^2}^2+\|\nabla u\|_{L^2}^2\Bigr)\\
+\|\sqrt t\nabla u_t\|_{L^2}^2+\frac16\|\sqrt\rho(u_t,\dot u)\|_{L^2}^2
+{c_\Omega}\|\nabla^2u,\nabla P\|_{L^2}^2 \\
\leq C\|\sqrt\rho v\|_{L^4}^4\|\nabla u\|_{L^2}^2+\int_\Omega t\div(\rho v) \dot u\cdot u_t\,dx
-\int_\Omega t\rho (v_t\cdot\nabla u)\cdot u_t\,dx.
\end{multline}
We integrate by parts in  the second term of the right-hand side, which gives
$$
\int_\Omega t\div(\rho v) \dot u\cdot u_t\,dx= -\int_\Omega t\bigl(\rho v\cdot\nabla\dot u\bigr)\cdot u_t\,dx-\int_\Omega t\bigl(\rho v\cdot\nabla u_t\bigr)\cdot\dot u\,dx.$$
The two integrals may be handled  as for proving \eqref{eq:weighttt}. We get
$$
\int_\Omega t\div(\rho v) \dot u\cdot u_t\,dx\leq \frac14\|\sqrt t(\nabla\dot u,\nabla u_t)\|_{L^2}^2
+C\|\sqrt\rho\, v\|_{L^4}^4\|\sqrt{\rho t}(\dot u, u_t)\|_{L^2}^2.$$
To bound the last term of \eqref{eq:H1half2}, we proceed as follows (for all $\eps>0$):
$$\begin{aligned}
\int_\Omega t\rho (v_t\cdot\nabla u)\cdot u_t\,dx&\leq \|\sqrt{\rho t} v_t\|_{L^2}\|\sqrt{\rho t} u_t\|_{L^4}\|\nabla u\|_{L^4}\\
&\leq \eps \|\nabla^2u\|_{L^2}^{2}+ \eps\|\sqrt{t} \nabla u_t\|_{L^2}^2+ C_\eps \|\sqrt{\rho t} v_t\|_{L^2}^2\|\sqrt{\rho t} u_t\|_{L^2}
\|\nabla u\|_{L^2}.
\end{aligned}$$
From the definition of $\dot u$ and \eqref{ed2}, it is easy to get 
\begin{equation}\label{eq:H1half4}
\|\sqrt t(\nabla^2 u,\nabla P,\sqrt\rho \dot u)\|_{L^2}\leq C\bigl(\|\sqrt{\rho t}\, u_t\|_{L^2} + \|\sqrt\rho\, v\|_{L^4}^2 
\|\sqrt t\nabla u\|_{L^2}\bigr)\cdotp\end{equation}
By H\"older inequality, \eqref{eq:GN} and \eqref{eq:stokesLp} with $p=4,$ we also notice that
$$
\|\sqrt t\nabla\dot u\|_{L^2}-\|\sqrt t\nabla u_t\|_{L^2}
\lesssim\|\sqrt t\nabla v\|_{L^4} \|\nabla u\|_{L^2}^{1/2}\|\nabla^2u\|_{L^2}^{1/2}
+\|v\|_{L^4}\|\sqrt{\rho t}\dot u\|_{L^2}^{1/2}\|\sqrt t\nabla\dot u\|_{L^2}^{1/2}
$$
which implies that
$$
\|\sqrt t\nabla\dot u\|_{L^2}\leq 2\|\sqrt t\nabla u_t\|_{L^2}
+\frac1{4}\|\nabla^2 u\|_{L^2}+ C\bigl(\|\sqrt t\nabla v\|_{L^4}^2\|\nabla u\|_{L^2}
+\|v\|_{L^4}^2\|\sqrt{\rho t}\dot u\|_{L^2}\bigr)\cdotp$$
Inserting all the above inequalities in \eqref{eq:H1half2}, then using Gronwall lemma
and \eqref{eq:weight3},  we discover that
\begin{multline}\label{eq:H1half5}
Y_1(t)\lesssim\|\nabla u_0\|_{L^2}^2e^{\wt C_1^v(t)}\with \wt C_1^v(t):=C\int_0^t\bigl(\|\sqrt\tau\nabla v, v\|_{L^4}^4+ \|\sqrt{\rho \tau} v_\tau\|_{L^2}^2\bigr)d\tau\\
\andf Y_1(t):= \|\sqrt{\rho t}(u_t,\dot u)\|_{L^2}^2+\|\nabla u\|_{L^2}^2
+\frac{c_\Omega}{\rho^*}\|\sqrt t(\nabla^2u,\nabla P)\|_{L^2}^2\hspace{3cm}\\
+\int_0^t\Bigl(\|\sqrt\tau(\nabla u_\tau,\nabla\dot u)\|_{L^2}^2
+\|\sqrt\rho(u_\tau,\dot u)\|_{L^2}^2+ \frac{c_\Omega}{\rho^*}\|\nabla^2u,\nabla P\|_{L^2}^2\Bigr)d\tau.\end{multline}

\subsubsection{Decay  estimates  with weight $t$:}
Still assuming $f\equiv0,$ we now take the $L^2$ scalar product of \eqref{eq:ddotu}  with  $tD_t(t\dot u)$
and get
$$\displaylines{\frac12\frac d{dt}\|\nabla(t\dot u)\|_{L^2}^2 + \|\sqrt \rho D_t(t\dot u)\|_{L^2}^2
\hfill\cr\hfill=\int_{\Omega} \bigl(tF-t\nabla\dot P+ \rho\dot u\bigr)\cdot D_t(t\dot u)\,dx
+\int_{\Omega} \Delta(t\dot u)\cdot(v\cdot\nabla(t\dot u))\,dx.}
$$
Hence for all $\varepsilon>0,$
\begin{multline}\label{eq:weight6}\frac12\frac d{dt}\|\nabla(t\dot u(t))\|_{L^2}^2 + \|\sqrt\rho\, D_t(t\dot u)\|_{L^2}^2
\leq\varepsilon  \bigl(\|\nabla^2(t\dot u)\|_{L^2}^2+\|\sqrt\rho\,D_t(t\dot u)\|_{L^2}^2\bigr)\\+ 
\frac1{\varepsilon}\Bigl(\|v\cdot\nabla(t\dot u)\|_{L^2}^2+\|\sqrt\rho\,\dot u\|_{L^2}^2+\Bigl\|\frac{t F-t\dot \nabla P}{\sqrt\rho}\Bigr\|_{L^2}^2\Bigr)\cdotp\end{multline}
To continue the computations, we need to estimate $t\dot P$ and $t\nabla^2\dot u.$
To this end, one can remember Inequality \eqref{eq:stokes3/2} and observe that 
$$\|\sqrt\rho t\ddot u\|_{L^2}\leq  \|\sqrt\rho\,D_t(t\dot u)\|_{L^2}+\|\sqrt\rho\,\dot u\|_{L^2}.$$
Hence, taking $\eps$ small enough in \eqref{eq:weight6} yields: 
\begin{multline}\label{eq:weight8}\|\nabla(t\dot u(t))\|_{L^2}^2 + \|\sqrt \rho D_{t}({t}\dot u),\nabla({t}\dot P),\nabla^2({t}\dot u)\|_{L^2}^2
\lesssim \|\sqrt\rho\,\dot u\|_{L^2}^2 \\+  \|v\cdot\nabla({t}\dot u)\|_{L^2}^2
+\| {t}\nabla^2v\otimes \nabla u\|_{L^2}^2+ \| {t}\nabla^2u\otimes \nabla v\|_{L^2}^2+\|{t}\nabla v\cdot\nabla P\|_{L^2}^2.\end{multline}
The first term of the right-hand side may be bounded according to \eqref{eq:H1ter}. So we are left with bounding all the other terms.
We have 
$$\begin{aligned}
\|v\cdot\nabla({t}\dot u)\|_{L^2}^2
&\leq\frac C\eps\|v\|_{L^4}^4\|\nabla({t}\dot u)\|_{L^2}^2+\eps\|\nabla^2({t}\dot u)\|_{L^2}^2
\\
\|{t}\nabla^2v\otimes \nabla u\|_{L^2}^2
&\lesssim \|{t}\nabla^2v\|_{L^4}^2\biggl(\|\nabla u\|_{L^2}^2\|\nabla^2 u\|_{L^2}^2\biggr)^{1/2}\\
\| {t}\nabla^2u\otimes \nabla v\|_{L^2}^2+\|{t}\nabla v\cdot\nabla P\|_{L^2}^2
&\lesssim  \|\sqrt {t}(\nabla^2 u,\nabla P)\|_{L^4}^2 \|\sqrt {t} \nabla v\|_{L^4}^2. \end{aligned}$$
Using regularity   estimates for \eqref{edu1}
and \eqref{eq:lad} yields
$$
\|\sqrt t(\nabla^2 u,\nabla P)\|_{L^4}^2 \lesssim \| \sqrt t\dot u\|_{L^4}^2 \lesssim 
\|\dot u\|_{L^2} \|t\nabla \dot u\|_{L^2}.
$$
Hence 
$$\begin{aligned}
\| {t}\nabla^2u\otimes \nabla v\|_{L^2}^2\!+\!\|{t}\nabla v\cdot\nabla P\|_{L^2}^2
&\lesssim  \|\sqrt {t}\nabla v\|_{L^4}^2  \|\dot u\|_{L^2} \|{t}\nabla \dot u\|_{L^2}\\
&\lesssim  \|\dot u\|_{L^2}^2 \!+\!  \|\sqrt {t}\nabla v\|_{L^4}^4 \|{t}\nabla \dot u\|_{L^2}^2.\end{aligned}$$
Plugging all these inequalities in \eqref{eq:weight6},  using \eqref{eq:H1ter} and integrating on $[0,t]$ gives
$$\displaylines{Y_2(t):=
\|\nabla(t\dot u(t))\|_{L^2}^2 + \int_0^t\|\sqrt \rho D_\tau(\tau\dot u),\nabla(\tau\dot P),\nabla^2(\tau\dot u)\|_{L^2}^2\,d\tau
\\\hfill\cr\hfill
\lesssim\!  \int_0^t\!\!\bigl(\|v\|_{L^4}^4\!+\!  \|\sqrt \tau\nabla v\|_{L^4}^4 \bigr) \|\tau\nabla \dot u\|_{L^2}^2d\tau
+ \|\nabla u_0\|_{L^2}^2e^{C\!\int_0^t\|v\|_{L^4}^4d\tau}\bigl(1+ \|\tau\nabla^2v\|_{L_t^4(L^4)}^4\bigr)\cdotp}$$
At this stage, Gronwall lemma enables us to conclude that 
\begin{equation}\label{eq:weight9} Y_2(t)\leq C  \|\nabla u_0\|_{L^2}^2e^{\wt C_2^v(t)}\with
\wt C_2^v(t):= C\!\int_0^t\|v,\sqrt \tau\nabla v,\tau\nabla^2v\|_{L^4}^4\,d\tau.\end{equation}

\subsubsection{Estimates in $\dot H^{s}$ for $s\in(0,1)$}
If we denote by $E$ the linear  operator that associates to $(u_0,g)$ the solution $u$ to \eqref{eq:LINS0} on $\R_+\times\Omega,$ 
then the previous inequalities \eqref{eq:L2} and \eqref{eq:H1ter}  and the fact that the norms in $L^2(\rho\,dx)$
or $L^2(dx)$ are equivalent (recall \eqref{eq:rho}) ensure that:
\begin{itemize}
\item $E$ maps  $L^2(\Omega)\times L^2(\R_+;\dot H^{-1}(\Omega))$ to $L^\infty(\R_+;L^2(\Omega))\cap L^2(\R_+;\dot H^1(\Omega))$;
\item $E$ maps  $\dot H^1(\Omega)\times L^2(\R_+;L^2(\Omega))$ to $L^\infty(\R_+;\dot H^1(\Omega))\cap L^2(\R_+;\dot H^2(\Omega)).$
\end{itemize}
Consequently, the  complex interpolation theory ensures that, for all $s\in[0,1],$  
 $$E : \dot H^s(\Omega)\times L^2(\R_+;\dot H^{s-1}(\Omega))\to L^\infty(\R_+;\dot H^s(\Omega))\cap L^2(\R_+;\dot H^{s+1}(\Omega))$$
with, for some constant $C_\rho$ depending only on $\rho_*$ and $\rho^*,$  the bound:
 \begin{multline}\label{eq:Hs}
\underset{t\in[0,T]}{\sup} \|u(t)\|_{\dot H^s}^2+
\int_0^T\|u\|_{\dot H^{s+1}}^2\,dt\\\leq C_\rho e^{Cs\rho^*\int_0^T\|\sqrt\rho\, v\|_{L^4}^4\,dt} \biggl(\|u_0\|_{\dot H^s}^2
+ \int_0^T\|g\|_{\dot H^{s-1}}^2\,dt\biggr)\cdotp\end{multline}
For $g\equiv0,$  due to \eqref{eq:weighttt}, \eqref{eq:weight9}, 
for all $t>0,$ the linear  operator   that associates to $u_0$ the function $t\dot u(t)$ with  $u$ being the solution to \eqref{eq:LINS0}
with no source term maps $L^2$ to $L^2$ and $\dot H^1$ to $\dot H^1.$
 Hence it maps $\dot H^s$ to $\dot H^s$ for all $s\in[0,1]$ and we have: 
 \begin{equation}\label{eq:weighttHs}
\|t\dot u(t)\|_{\dot H^s}\leq C e^{\frac s2\wt C_2^v(t)} \|u_0\|_{\dot H^s}\quad\hbox{for all }\ t>0.
\end{equation}


\subsection{Estimates in negative Sobolev spaces}

We here prove estimates for \eqref{eq:LINS0} in the case of initial data in Sobolev space with negative regularity.

\subsubsection{Data in $\dot H^{-1}$}

To estimate $\sqrt \rho\,u$ in $L^2(0,T\times\Omega),$  	 we consider  the following \emph{backward} parabolic system: 
	\begin{equation}\label{eq:backward}\left\{\begin{aligned}
 &\rho w_{t}+\rho v\cdot \nabla w + \Delta w+\nabla Q=  \rho u, \\
&\div w=0,\\ &w|_{t=T}=0.
\end{aligned}\right.\end{equation}
By definition of $w,$ we have 	
$$\int_0^T\!\!\!\int_{\Omega}  u\cdot (\rho u)\,dx\,dt 
= \int_0^T\!\!\!\int_{\Omega} u\cdot\bigl(\rho w_{t}+\rho v\cdot \nabla w + \Delta w+\nabla Q\bigr)\,dx\,dt.$$
	Integrating  by parts and remembering  that $\d_t\rho+\div(\rho v)=0$ and $\div w=0$ yields 
$$\displaylines{\int_0^T\!\!\!\int_{\Omega} \rho |u|^2\,dx\,dt 
= -	\int_0^T\!\!\!\int_{\Omega} \bigl (\rho\dot u -\Delta u+\nabla P\bigr)\cdot w\,dx\,dt
\hfill\cr\hfill+\int_{\Omega}\bigl( (\rho u)(T)\cdot w(T)-\rho_0 u_0\cdot w(0)\bigr)\,dx.}$$
	As $w(T)=0$ and $u$ satisfies \eqref{eq:LINS0}, we conclude that 
	$$\int_0^T\!\!\!\int_{\Omega}  \rho |u|^2\,dx\,dt = -\int_{\Omega}\rho_0 u_0\cdot w(0)\,dx\leq \|\rho_0 u_0\|_{\dot H^{-1}} \|\nabla w(0)\|_{L^2}.$$
Now, adapting  the proof of 
 \eqref{eq:H1ter}  to \eqref{eq:backward} yields  
$$
\|\nabla w(0)\|_{L^2}^2\leq  e^{\rho^*\int_0^T\|\sqrt\rho\, v\|_{L^4}^4\,dt}\|\sqrt\rho u\|_{L^2(0,T\times\Omega)}^2.$$
Hence we have 
\begin{equation}\label{eq:uL2L2}
\|\sqrt\rho\,u\|_{L^2(0,T\times\Omega)}\leq  \|\rho_0 u_0\|_{\dot H^{-1}} e^{\frac{\rho^*}2\int_0^T\|\sqrt\rho\,v\|_{L^4}^4\,dt}.\end{equation}
In order to bound $\cP(\rho u)(T)$ in $\dot H^{-1},$  we start from  $$
\|\cP(\rho u)(T)\|_{\dot H^{-1}}=\underset{\begin{smallmatrix}\|w_T\|_{\dot H^1}=1\\\div w=0\end{smallmatrix}}{\sup} \int_{\Omega} (\rho u)(T)\cdot w_T\,dx,$$
and solve \eqref{eq:backward} with no source term and data $w_T$ at time $t=T.$  Hence, 
$$\begin{aligned}
0&=\int_0^T\!\!\!\int_{\Omega}\bigl(\rho w_{t}+\rho v\cdot \nabla w + \Delta w+\nabla Q\bigr)\cdotp u\,dx\,dt\\
&=  -	\int_0^T\!\!\!\int_{\Omega} \rho(\d_tu+v\cdot\nabla u-\Delta u\bigr)\cdot w\,dx\,dt
+\int_{\Omega}\bigl( \rho(T) u(T)\cdot w_T-\rho_0 u_0\cdot w(0)\bigr)\,dx.\end{aligned}
$$
Since $u$ satisfies \eqref{eq:LINS0} and $\div w=0,$ we get
\begin{equation}\label{eq:uu}\int_{\Omega} (\rho u)(T)\cdot w_T\,dx = \int_{\Omega}\rho_0u_0\cdot w(0)\,dx.\end{equation} 
As
$$\|\nabla w(0)\|_{L^2} \leq e^{\frac{\rho^*}2\int_0^T\|\sqrt\rho\,v\|_{L^4}^4\,dt} \|\nabla w_T\|_{L^2},$$
we conclude that
\begin{equation}\label{eq:uH-1} 
\|\cP(\rho u)(T)\|_{\dot H^{-1}}\leq \|\cP(\rho_0 u_0)\|_{\dot H^{-1}}   e^{\frac{\rho^*}2\int_0^T\|\sqrt\rho\,v\|_{L^4}^4\,dt}\cdotp\end{equation}

\subsubsection{Estimates in $\dot H^{-s}$ for $s\in(0,1)$}

We start from:
 $$\|\cP(\rho u)(T)\|_{\dot H^{-s}}= \underset{\begin{smallmatrix}\|w_T\|_{\dot H^s}=1\\\div w=0\end{smallmatrix}}{\sup} 
 \int_{\Omega} (\rho u)(T)\cdot w_T\,dx.$$
Using \eqref{eq:uu}, we get for any  divergence free  $w_T\in\dot H^s$ with norm equal to $1,$
$$\left|  \int_{\Omega} (\rho u)(T)\cdot w_T\,dx\right|\leq \|\cP(\rho_0 u_0)\|_{\dot H^{-s}} \|w(0)\|_{\dot H^s},$$
where $w$ is the solution of \eqref{eq:backward} with no source term and data $w_T$ at time $T.$
\smallbreak
Keeping \eqref{eq:Hs} in mind, we easily conclude that 
\begin{equation}\label{eq:H-s}\|\cP(\rho u)(T)\|_{\dot H^{-s}}
\leq  C \|\cP(\rho_0 u_0)\|_{\dot H^{-s}}\: e^{\frac{Cs}2\rho^*\int_0^T\|\sqrt\rho\, v\|_{L^4}^4\,d\tau}.\end{equation}


\subsection{More time decay estimates}

In this paragraph, we point out a number of   time decay estimates for \eqref{eq:LINS0}
in Sobolev and Lebesgue spaces that may be deduced from what we proved hitherto and
basic interpolation results.

\subsubsection{Sobolev decay estimates}
They are summarized in the following proposition:
\begin{proposition}\label{p:decay1}  The following estimates hold: 
\begin{itemize}
\item 
For any $0\leq s\leq2$ and $0\leq s'\leq1,$  we have
\begin{equation}\label{eq:Hsdecay1}
\|u(t)\|_{\dot H^s}\leq C_{\rho,v}\,t^{-\frac{s+s'}2}\|\cP(\rho_0 u_0)\|_{\dot H^{-s'}},\qquad t>0.
\end{equation}
\item For any $0\leq s, s'\leq1,$ 
\begin{equation}\label{eq:Hsdecay2}
\|tu_t(t)\|_{\dot H^s}+ \|t\dot u(t)\|_{\dot H^s}\leq C_{\rho,v}\, t^{-\frac{s+s'}2}\|\cP(\rho_0 u_0)\|_{\dot H^{-s'}},\qquad t>0.
\end{equation}
\item For any $0\leq s\leq 1,$
\begin{align} \label{eq:H1decay1}
\|t\dot u(t), u(t)\|_{\dot H^1} &\leq Ce^{\wt C_2^v(t)+\wt C_3^v(t)}\,t^{\frac{s-1}2}\|u_0\|_{\dot H^s},\\
 \label{eq:H1decay2}
\|\dot u(t), u_t(t)\|_{L^2} &\leq Ce^{\wt C_2^v(t)+\wt C_3^v(t)}\,t^{-\frac{2-s}2}\|u_0\|_{\dot H^s},\\
 \label{eq:H1decay3}
\|\dot u(t)\|_{\dot H^s} &\leq Ce^{\wt C_2^v(t)+\wt C_3^v(t)}\,t^{-\frac{1+s}2}\|u_0\|_{\dot H^1}.
 \end{align}
  \end{itemize}
\end{proposition}
\begin{proof}
The previous parts  guarantee  that:
\begin{align}\label{eq:decay1}
t^{k/2}\|\nabla^k u(t)\|_{L^2} &\leq C_{\rho,v} \,\|u_0\|_{L^2} \quad\hbox{for }\ k=0,1,2,\\\label{eq:decay2}
t^{1+k/2}\|\nabla^k(u_t,\dot u)(t)\|_{L^2}   &\leq C_{\rho,v} \,\|u_0\|_{L^2} \quad\hbox{for }\ k=0,1.\end{align}
The key observation for proving  \eqref{eq:Hsdecay1} is that having the density bounded and bounded away from zero ensures
that 
\begin{equation}\label{eq:keyL2}\|\cP(\rho z)\|_{L^2}\simeq \|z\|_{L^2}\ \hbox{ for all }\ z\in L^2_\sigma.\end{equation}
 Indeed, since $\cP$ is a $L^2$ orthogonal projector we may write
$$
\|\cP(\rho z)\|_{L^2} \leq \|\rho z\|_{L^2}\leq \rho^* \|z\|_{L^2}$$
and 
$$\rho_*\|z\|_{L^2}^2\leq \int_\Omega \rho |z|^2\,dx =\int_\Omega  \cP(\rho z)\cdot z\,dx \leq \|\cP(\rho z)\|_{L^2}\|z\|_{L^2}.$$
Inequality \eqref{eq:Hsdecay1} in the case $s'=0$ thus follows from \eqref{eq:decay1} 
with $k=0,2$ and complex interpolation. 
In order to attain  negative values of $s',$ we use again \eqref{eq:keyL2}, then  argue by duality as follows for all $t>0$:
$$\begin{aligned}\|\cP(\rho u)(t)\|_{L^2} &= \underset{\|w\|_{L^2_\sigma}=1}{\sup}\int_{\Omega} (\rho u)(t)\cdot w\,dx\\
&= \underset{\|w\|_{L^2_\sigma}=1}{\sup}\int_{\Omega} \rho_0 u_0\cdot w(0)\,dx\\
&\leq \|\cP(\rho_0 u_0)\|_{\dot H^{-s'}}\ \underset{\|w\|_{L^2_\sigma}=1}{\sup}\|w(0)\|_{\dot H^{s'}},
\end{aligned}
$$
where $w(0)$ stands for the solution at time $0$ of the backward Stokes system \eqref{eq:backward} with no source
term and data $w$ at time $t.$ 
Now, using  the inequality we have just proved (that, obviously, also  holds true for \eqref{eq:backward}), we discover that
$$\|w(0)\|_{\dot H^{s'}}\leq C t^{-s'/2}\|w\|_{L^2},$$
whence: 
\begin{equation}\label{eq:Hsdecay1a} \|\rho(t)u(t)\|_{L^2}\leq C t^{-s'/2} \|\cP(\rho_0 u_0)\|_{\dot H^{-s'}}.\end{equation}
Since Inequality \eqref{eq:decay1} is valid on any interval $[t_0,t]$ (if replacing $u_0$ by $u(t_0)$ and $t$ by $t-t_0,$ of course), 
one can assert that for all $s\in[0,2],$ we have
$$\|u(t)\|_{\dot H^s}\leq Ct^{-\frac s2} \|(\rho u)(t/2)\|_{L^2},$$
which, combined with  
\eqref{eq:Hsdecay1a} (at time $t/2$) completes the proof of \eqref{eq:Hsdecay1} 
for all $0\leq s\leq2$ and $0\leq s'\leq1.$ 
\smallbreak
Next, using \eqref{eq:decay2} 
with $k=0,1$ and complex interpolation yields \eqref{eq:Hsdecay2} for $s'=0$ and all $s\in[0,1].$
Since the inequality also holds true if $u_0$ is replaced with $u(t/2),$ using again \eqref{eq:Hsdecay1a}
yields the desired inequality for all $s'\in[0,1].$
\smallbreak
By the same token, combining the above result  with  the continuity properties resulting from 
Inequalities \eqref{eq:weight3/2}, \eqref{eq:H1ter}, \eqref{eq:H1half5}
and  \eqref{eq:weight9} gives the last three inequalities of the statement. The details are left to the reader.
\end{proof}

\subsubsection{Decay estimates in Lebesgue spaces}
Inequalities \eqref{eq:decay1} and \eqref{eq:decay2} also imply the following result.
\begin{proposition}\label{p:decay2}   
 The following inequalities hold true:
\begin{itemize}
\item  If  $1<p\leq 2\leq q\leq\infty$ then 
\begin{equation}\label{eq:decayLp1} 
\|u(t)\|_{L^q} + \|\sqrt t\, \nabla u(t)\|_{L^q} \leq C_{\rho,v}
t^{\frac1q-\frac1p}\|u_0\|_{L^p}.
\end{equation}
\item  If  $1<p\leq 2\leq q<\infty$ then
\begin{equation}\label{eq:decayLp2} 
\|t(\dot u,u_t,\nabla^2 u,\nabla P)(t)\|_{L^q} \leq 
C_{\rho,v} t^{\frac1q-\frac1p}\|u_0\|_{L^p}.
\end{equation}\end{itemize}
\end{proposition}
\begin{proof}
Combining Gagliardo-Nirenberg inequality  \eqref{eq:GN} and \eqref{eq:decay1} with $k=0,1,2,$ 
it is easy to get: 
\begin{equation}\label{eq:decay5}
\|u(t)\|_{L^q} +\|\sqrt t\nabla u(t)\|_{L^q} \leq C_{\rho,v} t^{\frac1q-\frac12}\|u_0\|_{L^2},\quad 2\leq q<\infty\end{equation}
while \eqref{eq:decay2} ensures that
\begin{equation}\label{eq:decay5a}\|u_t(t),\dot u(t)\|_{L^q} \leq C_{\rho,v} t^{\frac1q-\frac32}\|u_0\|_{L^2}.\end{equation}
Since $(u,\nabla P)$ satisfies the Stokes system \eqref{edu1}, 
Inequality \eqref{eq:stokesLp} gives 
\begin{equation}\label{eq:stokes0}\|\nabla^2 u(t)\|_{L^q}+\|\nabla P(t)\|_{L^q}\leq C_{\rho,v} t^{\frac1q-\frac32}\|u_0\|_{L^2}, \quad 2\leq q<\infty.\end{equation}
 Remember that\footnote{In the torus case, this inequality holds under the assumption $\int_{\T^2} az\,dx=0$ 
 for some nonnegative function $a$ with mean value $1.$ The idea of the proof is similar to that of \eqref{eq:GNT}.}
 \begin{equation}\label{eq:GN2} \|z\|_{L^\infty}\leq C \|z\|_{L^4}^{1/2}\|\nabla z\|_{L^4}^{1/2}.\end{equation}
Taking first $z=u$ and using \eqref{eq:decay5} with $p=4,$ 
then $z=\nabla u$ and using  \eqref{eq:stokes0} with $p=4$ allows 
 to reach the index $q=\infty$ in \eqref{eq:decay5}. 
 
 In \eqref{eq:decay5} and \eqref{eq:stokes0}, the term $\|u_0\|_{L^2}$
 may be replaced with $\|u(t/2)\|_{L^2}.$
 Consequently, using \eqref{eq:rho0},    \eqref{eq:Hsdecay1a}, embedding $L^p\hookrightarrow \dot H^{-1+2/p}$
 for all $1<p\leq 2$ and the fact that $\cP:L^p\to L^p$ ensures that  
 $$\begin{aligned}
 \|u(t)\|_{L^2}\simeq \|\cP(\rho u)(t)\|_{L^2}&\leq C_{\rho,v}t^{\frac12-\frac1p}\|\cP(\rho_0u_0)\|_{\dot H^{1-\frac2p}}\\
& \leq C_{\rho,v}t^{\frac12-\frac1p}\|\cP(\rho_0 u_0)\|_{L^p} \leq C_{\rho,v} t^{\frac12-\frac1p}\|u_0\|_{L^p}\end{aligned}
 $$
 which, plugged into \eqref{eq:decay5} and \eqref{eq:stokesLp} completes
 the proof of  \eqref{eq:decayLp1} and of \eqref{eq:decayLp2} for all
 admissible values of $p$ and $q.$  
\end{proof}

\subsubsection{Decay estimates for $L^2$-in-time norms}

Putting together \eqref{eq:L2}, \eqref{eq:weight3}, \eqref{eq:weighttt} and \eqref{eq:weight3/2},  we see  that
\begin{multline}\label{eq:decay3b}
\int_0^t\bigl(\|\nabla u\|_{L^2}^2+\|\sqrt\tau(\nabla^2u,\nabla P)\|_{L^2}^2
+\|\sqrt\tau(\dot u,u_\tau)\|_{L^2}^2\\+\|\tau(\nabla u_\tau,\nabla\dot u)\|_{L^2}^2
+\|\tau^{3/2}\ddot u\|_{L^2}^2 + \|\tau^{3/2}(\nabla^2\dot u,\nabla \dot P)\|_{L^2}^2\bigr)d\tau
\leq C_{\rho,v} \|u_0\|_{L^2}^2.
\end{multline}
This will enable us to prove the following family of decay estimates:
\begin{proposition}\label{p:decay3} 
The following inequalities hold true:
 \begin{align}
 \label{eq:decay8}
\|\tau^{\frac12-\frac1q}\nabla u\|_{L_t^2(L^q)} &\leq C_{\rho,v} \|u_0\|_{L^2}\quad\hbox{for all }\ 2\leq q\leq\infty,\\
 \label{eq:decay9}\|\tau^{1-\frac1q}(\dot u,u_t)\|_{L_t^2(L^q)} &\leq C_{\rho,v} \|u_0\|_{L^2}\quad\hbox{for all }\ 2\leq q\leq \infty,\\
\label{eq:decay10}
\|\tau^{1-\frac1q}(\nabla^2u,\nabla P)\|_{L_t^2(L^q)} &\leq C_{\rho,v} \|u_0\|_{L^2}\quad\hbox{for all }\ 2\leq q<\infty,\\\label{eq:decay11}
  \|\tau^{\frac32-\frac1q}\nabla\dot u\|_{L^2_t(L^q)} &\leq C_{\rho,v} \|u_0\|_{L^2}\quad\hbox{for all }\ 2\leq q<\infty.\end{align}
\end{proposition}
\begin{proof}
Except for  $q=\infty,$ Inequality \eqref{eq:decay8}  follows from Gagliardo-Nirenberg inequality \eqref{eq:GN} 
and  the fact that
$$ \|\nabla u\|_{L^2_t(L^2)} +\|\sqrt \tau\nabla^2u\|_{L^2_t(L^2)}\leq C_{\rho,v}\|u_0\|_{L^2}.$$
Similarly, except for the case $q=\infty,$ Inequality \eqref{eq:decay9} for $\dot u$ stems from \eqref{eq:GN} and 
$$ \|\tau\nabla \dot u\|_{L^2_t(L^2)} +\|\sqrt \tau\dot u\|_{L^2_t(L^2)}\leq C_{\rho,v} \|u_0\|_{L^2}.$$
Now, since $(u,P)$ satisfies \eqref{edu1},  the regularity properties of the Stokes system pointed out in
\eqref{eq:stokesLp}, and \eqref{eq:decay9} guarantee  that 
 $$
\|\tau^{1-\frac1q}(\nabla^2u,\nabla P)\|_{L_t^2(L^q)} \leq C_{\rho,v} \|u_0\|_{L^2}\quad\hbox{for all }\ 2\leq q<\infty.$$
Putting together this latter inequality  and \eqref{eq:decay8} with $q=4,$ and 
remembering  \eqref{eq:GN2} yields \eqref{eq:decay8}  for $q=\infty.$
\smallbreak
Note that \eqref{eq:decay3b} also implies that
$$ \|\tau^{3/2}\nabla^2 \dot u\|_{L^2_t(L^2)} +\| \tau\nabla \dot u\|_{L^2_t(L^2)}\leq C_{\rho,v} \|u_0\|_{L^2},$$
and thus \eqref{eq:decay11}, by \eqref{eq:GN}.  Using it with $q=4$ as well as \eqref{eq:decay9} (also 
with $q=4$) and \eqref{eq:GN2} gives \eqref{eq:decay9} for $\dot u$ and  $q=\infty.$
\smallbreak To  prove that  $u_t$ satisfies \eqref{eq:decay9}, it suffices to check  that 
$$\|\tau^{1-\frac1q}\,v\cdot\nabla u\|_{L_t^2(L^q)}\leq C_{\rho,v} \|u_0\|_{L^2}\quad\hbox{for all }\ 2\leq q\leq\infty.$$
Now, by H\"older inequality, we have
$$\|\tau^{1-\frac1q}\,v\cdot\nabla u\|_{L_t^2(L^q)}\leq \|\tau^{\frac12}v\|_{L_t^\infty(L^\infty)}
\|\tau^{\frac12-\frac1{q}}\nabla u\|_{L_t^2(L^{q})}. $$
The term with $v$ is energy-like  (see \eqref{eq:decayLp1}),
which  completes the proof.
\end{proof}


\subsection{The Lipschitz control  and other properties needed for stability}

In the present subsection, we point out some additional properties of the velocity field that are valid 
in the case where $u_0$ is in $\wt B^0_{\rho_0,1}.$ The most important one is the Lipschitz control. 
We shall also prove  that the  regularity $\wt B^0_{\rho_0,1}$  is preserved by the flow, and that other norms 
that will be needed in the proof of uniqueness and stability are finite. 

These results follow from the Sobolev estimates we proved in the previous part and on the dynamic interpolation argument
presented for the heat equation in Section~\ref{s:results}. 

Now, fix some $u_0$ in $\wt B^0_{\rho_0,1}$ and a sequence $(u_{0,j})_{j\in\Z}$   of $L^2_\sigma$ such that
 \begin{multline}\label{def:abstractbesov}
 u_0=\sum_{j\in\Z} u_{0,j}\with \cP(\rho_0u_{0,j})\in \dot H^{-1/2},\ u_{0,j}\in \dot H^{1/2}\quad\hbox{for all }\ j\in\Z,
 \\\andf
\sum_{j\in\Z}\bigl(2^{-j/2} \|u_{0,j}\|_{\dot H^{1/2}} + 2^{j/2} \|\cP(\rho_0 u_{0,j})\|_{\dot H^{-1/2}}\bigr) \leq2\|u_0\|_{\wt B^0_{\rho_0,1}}.
\qquad\qquad
 \end{multline}
Then, for each $j\in\Z,$ we solve the linear system  
\begin{equation}\label{eq:LINSj}
\left\{\begin{aligned}
 &\rho \d_tu_{j}+\rho v\cdot \nabla u_j- \Delta u_j+\nabla P_j=0, \\
&\div u_j=0,\\
&u_j|_{t=0}=u_{0,j}.
\end{aligned}\right.
\end{equation}
{}From \eqref{def:abstractbesov} and the uniqueness properties of System \eqref{eq:LINS0} in  the energy space, we deduce that 
\begin{equation}\label{eq:decompouj}u=\sum_{j\in\Z} u_j.\end{equation}


\subsubsection{The Lipschitz bound}

Recall  the following  Gagliardo-Nirenberg inequality:
  \begin{equation}\label{eq:GN3}
\|\nabla z\|_{L^\infty}\leq C\|z\|_{L^4}^{1/4}\|\nabla^2 z\|_{L^4}^{3/4}.
\end{equation}
Combined with  the elliptic estimates for the Stokes system  and Sobolev embedding, this implies that for all $t>0$ and $j\in\Z,$
$$\|\nabla u_j(t)\|_{L^\infty} \leq Ct^{-3/4}\|u_j(t)\|_{L^4}^{1/4}\|t\dot u_j(t)\|_{L^4}^{3/4}
\leq  Ct^{-3/4}\|u_j(t)\|_{\dot H^{1/2}}^{1/4}\|t\dot u_j(t)\|_{\dot H^{1/2}}^{3/4}.$$
Hence, taking advantage of \eqref{eq:Hs} and of \eqref{eq:weighttHs} gives
$$\|\nabla u_j(t)\|_{L^\infty} \leq C_{\rho,v} t^{-3/4} \|u_{0,j}\|_{\dot H^{1/2}}.$$
 Since we also have 
$$\|\nabla u_j(t)\|_{L^\infty} \leq C_{\rho,v} t^{-3/4} \|u_j(t/2)\|_{\dot H^{1/2}},$$
we conclude in light of \eqref{eq:Hsdecay1}  that
$$\|\nabla u_j(t)\|_{L^\infty} \leq C_{\rho,v} t^{-5/4} \|\cP(\rho_0u_{0,j})\|_{\dot H^{-1/2}}.$$
Hence arguing as  in Section \ref{s:results}, we conclude that
   \begin{equation}\label{eq:lipfinalter}
 \int_0^\infty\|\nabla u\|_{L^\infty}\,dt\leq C_{\rho,v}\, \|u_0\|_{\wt B^0_{\rho_0,1}}.\end{equation}
\begin{remark}  Recall  the  following  more accurate interpolation inequality:
\begin{equation}\label{eq:GN3bis}
\|\nabla z\|_{\dot B^{1/2}_{4,1}}\leq C\|z\|_{L^4}^{1/2}\|\nabla^2 z\|_{L^4}^{3/4}.
\end{equation}
Repeating the above dynamic interpolation procedure thus actually gives
$$\int_0^\infty\|\nabla u\|_{\dot B^{1/2}_{4,1}}\,dt\leq C_{\rho,v}\, \|u_0\|_{\wt B^0_{\rho_0,1}}.$$
Since  $\dot B^{1/2}_{4,1}\hookrightarrow \cC_b,$ this ensures that the flow of the velocity field is uniformly $C^1$
with respect to the space variable. 
\end{remark}

\subsubsection{Propagating the initial regularity}

 Owing to  \eqref{eq:Hs} and to \eqref{eq:H-s} with $s=1/2,$ we  have  for all $j\in\Z$ and $t\geq0,$
 $$\|u_j(t)\|_{\dot H^{1/2}}\leq C_{\rho,v} \|u_{0,j}\|_{\dot H^{1/2}}\andf 
   \|\cP(\rho u_j)(t)\|_{\dot H^{-1/2}}\leq C_{\rho,v}\|\cP(\rho_0u_{0,j})\|_{\dot H^{-1/2}}.$$
 Hence, multiplying the first (resp. second)  inequality by $2^{-j/2}$ (resp. $2^{j/2}$), then summing up on $j\in\Z$ yields 
 $$ \|u(t)\|_{\wt B^0_{\rho(t),1}} \leq C_{\rho,v} \|u_0\|_{\wt B^0_{\rho_0,1}}.$$

 \subsubsection{Additional  bounds for the pressure and the time derivative of the velocity}
 
 In addition to the Lipschitz bound on velocity,  our proof of uniqueness will require that  
  $\sqrt t \dot u$ and $\sqrt t \nabla P$  are
 in   $L^{4/3}(\R_+;L^4),$  and we will  also need the property that $\dot u$ and $\sqrt tD\dot u$
 are in $L^1(\R_+;L^2)$ for proving the stability of the flow map. 
  
 \smallbreak
 Again, in light of the decomposition \eqref{eq:decompouj} and of the triangle inequality, in order to prove that  
 $\sqrt t \dot u$ is in   $L^{4/3}(\R_+;L^4),$ it suffices to  estimate $t\dot u_j$ for all $j\in\Z.$ 
Now, owing to Sobolev embedding and the following inequalities  (that stem from  \eqref{eq:weighttHs} and \eqref{eq:Hsdecay2} with $s=s'=1/2$):
$$\|\dot u_j(t)\|_{\dot H^{1/2}}\leq C_{\rho,v} t^{-1}\|u_{0,j}\|_{\dot H^{1/2}}\andf
\|\dot u_j(t)\|_{\dot H^{1/2}}\leq C_{\rho,v}t^{-3/2}\|\cP(\rho_0u_{0,j})\|_{\dot H^{-1/2}},$$
we may write for all $A_j>0,$
$$\begin{aligned}
\|\sqrt t\dot u_j&\|_{L^{4/3}(\R_+;L^4)}^{4/3}\leq C\int_0^\infty t^{2/3}\|\dot u_j\|_{\dot H^{1/2}}^{4/3}\,dt\\
&\leq C_{\rho,v}\biggl(\int_0^{A_j} \!\! t^{2/3}(t^{-1}\|u_{0,j}\|_{\dot H^{1/2}})^{4/3}dt+
\int_{A_j}^\infty \!\! t^{2/3}(t^{-3/2}\|\cP(\rho_0u_{0,j})\|_{\dot H^{-1/2}})^{4/3}dt\biggr)\\
&\leq C_{\rho,v}\bigl(A_j^{1/3}\|u_{0,j}\|_{\dot H^{1/2}}^{4/3}+ A_j^{-1/3}\|\cP(\rho_0 u_{0,j})\|_{\dot H^{-1/2}}^{4/3}\bigr),
\end{aligned}$$
which gives, if taking $A_j=2^{-2j}$ and using  \eqref{eq:stokesLp}, 
\begin{equation}\label{eq:Puniq}
\|(\sqrt t\dot u, \sqrt t\nabla^2u,\sqrt t\nabla P)\| _{L^{4/3}(\R_+;L^4)} \leq C_{\rho,v}\,  \|u_0\|_{\wt B^{0}_{\rho_0,1}}.
\end{equation}
Similarly, in order to bound $\dot u$ in $L^1(\R_+;L^2),$ it suffices to get appropriate bounds in terms of the data
for $\dot u_j$ in $L^1(\R_+;L^2),$ for all $j\in\Z.$ 
The following inequalities (that stem from \eqref{eq:weighttt} and \eqref{eq:H1half5}): 
$$
\|\dot u_j(t)\|_{L^2} \leq   C_{\rho,v}\,t^{-1}\|u_{0,j}\|_{L^2}\andf  \|\dot u_j(t)\|_{L^2}\leq  C_{\rho,v}\,t^{-1/2}\|\nabla u_{0,j}\|_{L^2}$$
and complex interpolation give
$$ \|\dot u_j(t)\|_{L^2}\leq  C_{\rho,v}\,t^{-3/4}\|u_{0,j}\|_{\dot H^{1/2}}.$$
Furthermore, combining with  \eqref{eq:Hsdecay2},  we discover that for all $j\in\Z$:
$$\|\dot u_j(t)\|_{L^2} \leq  C_{\rho,v}\, t^{-5/4}\|\cP(\rho_{0} u_{0,j})\|_{\dot H^{-1/2}}.$$
Hence  we have  for all $j\in\Z$ and $A_j>0,$
$$\begin{aligned}
\int_0^\infty\|\dot u_j(t)\|_{L^2}\,dt &\leq \int_0^{A_j} \|\dot u_j(t)\|_{L^2}\,dt +\int_{A_j}^\infty\|\dot u_j(t)\|_{L^2}\,dt \\
&\leq C_{\rho,v}\biggl(\int_0^{A_j}\! \bigl(t^{-3/4} \|u_{0,j}\|_{\dot H^{1/2}}\bigr)dt
+\int_{A_j}^\infty\!\bigl(t^{-5/4} \|\cP(\rho_0u_{0,j})\|_{\dot H^{-1/2}}\bigr)dt\biggr)\\
&\leq C_{\rho,v}\biggl(A_j^{1/4} \|u_{0,j}\|_{\dot H^{1/2}}+A_j^{-1/4} \|\cP(\rho_0u_{0,j})\|_{\dot H^{-1/2}}\biggr)\cdotp\end{aligned}$$
Taking $A_j=2^{-2j},$  summing up on $j$ then using  the regularity properties of the Stokes system 
thus  gives 
\begin{equation}\label{eq:tuL1L1} \|\nabla^2u,\nabla P,\dot u\|_{L^1(\R_+;L^2)} \leq  C_{\rho,v}\, \|u_0\|_{\wt B^0_{\rho_0,1}}.\end{equation}
In the same way, one can prove that 
\begin{equation}\label{eq:tuL1L1b} \|\sqrt tD\dot u\|_{L^1(\R_+;L^2)} \leq  C_{\rho,v}\,
 \|u_0\|_{\wt B^0_{\rho_0,1}}.\end{equation}
 It suffices to use that, as a consequence of \eqref{eq:Hsdecay2} and \eqref{eq:H1decay1},  we have
 $$\begin{aligned} \|\sqrt t\nabla\dot u_j(t)\|_{L^2} &\leq  C_{\rho,v}t^{-3/4}\|u_{0,j}\|_{\dot H^{1/2}}\\
 \andf  \|\sqrt t\nabla\dot u_j(t)\|_{L^2}&\leq  
 C_{\rho,v}t^{-5/4}\|\cP(\rho_0 u_{0,j})\|_{\dot H^{-1/2}}.\end{aligned}$$


\section{A global well-posedness result for large data}   \label{s:proof}

This section is devoted to the proof of Theorem \ref{thm:2} and of stability estimates. 

\subsection{The proof of existence}

Consider data $(\rho_0,u_0)$ satisfying the hypotheses of Theorem \ref{thm:2}. 
Since  the space $\wt B^0_{\rho_0,1}$ is embedded in $L^2_\sigma,$  Theorem \ref{thm:1} provides us with  
a global weak solution $(\rho,u,\nabla P)$ satisfying the properties therein, and 
it is only a matter of checking that this solution has the additional properties that are listed in Theorem \ref{thm:2}. 
To do so, we fix some decomposition $\sum_ju_{0,j}$ of $u_0$ given by Definition \ref{def:espacepourri} and look, for all $j\in\Z,$
at the  solution $u_j$ to the linear system \eqref{eq:LINS0} with density $\rho,$ transport field $u$ and initial data $u_{0,j}.$
Since each  $u_{0,j}$ is in $L^2_\sigma\cap \dot H^{1/2}$ and $\cP(\rho_0 u_{0,j})\in \dot H^{-1/2},$ standard techniques yield a unique 
global solution  $(u_j,\nabla P_j)$  that satisfies for all $t\geq0,$
\begin{align}\label{eq:uj1}
\frac12\|\sqrt{\rho(t)}\,u_j(t)\|_{L^2}^2+\int_0^t&\|\nabla u_j\|_{L^2}^2\,d\tau= \frac12\|\sqrt{\rho_0}\, u_{0,j}\|_{L^2}^2,\\
\|\cP(\rho u_j)(t)\|_{\dot H^{-1/2}} &\leq C(\rho_*,\rho^*,\|u_0\|_{L^2})\|\cP(\rho_0 u_{0,j})\|_{\dot H^{-1/2}},\label{eq:uj2}\\
\|u_j(t)\|_{\dot H^{1/2}} &\leq C(\rho_*,\rho^*,\|u_0\|_{L^2})\|u_{0,j}\|_{\dot H^{1/2}}.\label{eq:uj3}
\end{align}
Remembering \eqref{eq:L2zj}, this ensures that the $L^2$-valued series 
$\sum_j u_j$ converges normally on $\R_+.$ Its sum  $\wt u$ thus
 also belongs to the energy space. Furthermore, as for each $j\in\Z,$ we have 
 $u_j\in \cC(\R_+;L^2)$ (observe that $t^{3/4} u_t^j$ is in $L^\infty(\R_+;L^2)$
 owing to \eqref{eq:H1decay2}), we deduce that $\wt u\in \cC(\R_+;L^2).$
 Next,   if we denote $u^n:=\sum_{|j|\leq n} u_j,$ then we see that for all $n\in\N,$
 $$\d_t (\rho(u^n-\wt u)) + \div(\rho u\otimes(u^n-\wt u)) -\Delta(u^n-\wt u)+\nabla(P^n-\wt P)=0,\qquad \div(u^n-\wt u)=0,$$
 which implies 
 $$\frac12\|\sqrt{\rho(t)}\,(u^n-\wt u)(t)\|_{L^2}^2+\int_0^t\|\nabla(u^n-\wt u)\|_{L^2}^2\,d\tau
 = \frac12\|\sqrt{\rho_0}\, (u^n(0)-u(0))\|_{L^2}^2.$$
 As the right-hand side tends to $0$ for $n$ going to $0,$ the velocity field $\wt u$
 satisfies the energy balance \eqref{eq:L2INS}, 
 and it is also easy to conclude that, like $u,$
it satisfies  \eqref{eq:LINS0} with density $\rho,$ transport field $u$ and initial data $u_{0}.$
In particular,
$$\d_t (\rho(u-\wt u)) + \div(\rho u\otimes(u-\wt u)) -\Delta(u-\wt u)+\nabla(P-\wt P)=0,\qquad \div(u-\wt u)=0.$$
As  $(u-\wt u)(0)=0,$ and the two solutions are in the energy space, they must coincide. 
Now,  Inequalities \eqref{eq:uj2} and \eqref{eq:uj3} ensure that
one can propagate the regularity $\wt B^0_{\rho_0,1},$ getting \eqref{eq:utB}.  
Likewise, justifying that $u$ satisfies \eqref{eq:Lip},  that $(\dot u, \sqrt t D\dot u, D^2u,\nabla P)\in L^1(\R_+;L^2)$
and that $\sqrt t\dot u\in  L^{4/3}(\R_+;L^4)$
may be achieved by following the arguments of the previous section. 
The fundamental point is that all the bounds that are needed for the $u_j$'s in the process only 
depend on  $\rho_*,$ $\rho^*,$ $\|u_0\|_{L^2},$ $\|\cP(\rho_0u_{0,j})\|_{\dot H^{-1/2}}$ and $\|u_{0,j}\|_{\dot H^{1/2}}.$


\subsection{The proof of uniqueness}

Let  $(\rho^1,u^1,\nabla P^1)$ and $(\rho^2,u^2,\nabla P^2)$ be two solutions
 fulfilling the properties listed in Theorem \ref{thm:2}, and corresponding to data $(\rho_{0}^1,u_{0}^1)$
 and  $(\rho_{0}^2,u_{0}^2),$ respectively.
  As in \cite{DM1}, in order to prove that  $(\rho^1,u^1,\nabla P^1)\equiv(\rho^2,u^2,\nabla P^2)$
  in the case where the two initial data coincide,  
    we shall compare the  solutions at the level of their own Lagrangian coordinates. 
    To do so,  we consider for $i=1,2,$ the flow $X^i$ of $u^i$ that is defined by 
    the following (integrated) ODE:
\begin{equation}\label{eq:ODE} X^i(t,y)=y+\int_0^tu^i(\tau,X^i(\tau,y))\,d\tau.\end{equation}
Since  $\nabla u^i$ is in $L^1(\R_+;L^\infty)$  and $\sqrt tu^i$ is in $L^\infty(0,T\times\Omega)$ (see \eqref{eq:decayLp1}
with $p=2$ and $q=\infty$), there exists  
 a unique continuous flow $X^i$ on $(0,T)\times\Omega,$ that is Lipschitz with respect to the space variable.  
 \smallbreak
 In Lagrangian coordinates the density  is  equal to the initial density. 
As for the velocity and the pressure, defined by 
 \begin{equation}\label{def:lag}
 Q^i(t,y)=P^i(t,X^i(t,y))\andf  v^i(t,y)=u^i(t,X^i(t,y)),
\end{equation} they satisfy
\begin{equation}\label{u1}
 \left\{\begin{array}{l}
  \rho_0^i v_{t}^i -\div_{\!v^i}\nabla_{v^i}v^i+ \nabla_{v^i}Q^i =0,\\[1ex] \div_{\!v^i}v^i=0,  
 \end{array}\right.
\end{equation}
where $\nabla_{v^i}:=(A^{i})^\top \nabla_y$ and 
$\div_{\!v^i}:= \div_{\!y}(A^{i}\cdot)= (A^{i})^\top:\nabla_y$ with $A^{i}:=(DX^{i})^{-1}.$  
The fact that  $\nabla u^i$ is in $L^1(\R_+;L^\infty)$ and the other properties of regularity ensure
that  (INS) and \eqref{u1} (with time independent density) are equivalent.
\smallbreak
Observe that, due to \eqref{eq:ODE} and to the definition of $v^i,$ we have
\begin{equation}\label{eq:DX}
DX^i(t,y)={\rm Id}+\int_0^t Dv^i(\tau,y)\,d\tau.
\end{equation}
Hence, since $\det DX^i\equiv1$ (owing to $\div v^i=0$),  we have for $i=1,2$:
\begin{equation}\label{eq:Ai}
A^i(t)= {\rm Id} +\begin{pmatrix} \int_0^t\d_2v^{i,2}\,d\tau&-\int_0^t\d_2v^{i,1}\,d\tau\\[1ex]
-\int_0^t\d_1v^{i,2}\,d\tau&\int_0^t\d_1v^{i,1}\,d\tau\end{pmatrix}\cdotp
\end{equation}
Hence $\dA:=A^2-A^1$ depends linearly on $\nabla \dv$  (with $\dv:=v^2-v^1$) as follows:
\begin{equation}\label{def:dA}
\dA(t)= \begin{pmatrix} \int_0^t\d_2\dv^2\,d\tau&-\int_0^t\d_2\dv^1\,d\tau\\[1ex]
-\int_0^t\d_1\dv^2\,d\tau&\int_0^t\d_1\dv^1\,d\tau\end{pmatrix}\cdotp
\end{equation}
Now, setting $\Delta_{v^i}:=\div_{\!v^i}\nabla_{v^i}$   and $\dQ:=Q^2-Q^1,$ we 
discover that $(\dv,\dQ)$ satisfies: 
\begin{equation}\label{eq:dv}\left\{\begin{array}{l}
\rho_{0}^1\dv_t-\Delta_{v^1}\dv +\nabla_{v^1}\dQ = \bigl(\Delta_{v^2}-\Delta_{v^1}\bigr)v^2
-(\nabla_{v^2}-\nabla_{v^1})Q^2-\dr_0\,v_{t}^2,\\[1ex]
\div_{\!v^1}\dv= (\div_{\!v^1}-\div_{\!v^2})v^2 =-\div (\dA v^2).\end{array}\right.\end{equation}
In order to prove uniqueness in the case where the initial data are the same and, 
more generally, stability estimates with respect to the initial data, using  the basic
 energy method  consisting in taking the $L^2$ scalar product of   \eqref{eq:dv} with $\dv$ is 
 not appropriate since one cannot eliminate the pressure term (there is no reason why we should have $\div_{\!v^1}\dv= 0$).  To overcome the difficulty,
 we proceed as in \cite{DM1}, solving first  the equation
 \begin{equation}\label{eq:div}
 \div_{\!v^1} w=-\div (\dA v^2) =- \dA^\top:\nabla v^2 \with   \dA:=A^2-A^1,\end{equation} 
 Then, we  look at the  system for $z:=\dv-w,$ namely:
 \begin{equation}\label{eq:w}
 \left\{\!\begin{array}{l}
\rho_0^1z_t -\Delta_{v^1}z +\nabla_{v^1}\dQ = \bigl(\Delta_{v^2}\!-\!\Delta_{v^1}\bigr) v^2\\
\hspace{4cm}-(\nabla_{v^2}\!-\!\nabla_{v^1})Q^2-\rho_0^1 w_t+ \Delta_{v^1}w -\dr_0\,v_{t}^2,\\[1ex]
\div_{\!v^1}z= 0,\end{array}\right.
 \end{equation}
 supplemented with $z|_{t=0}=\dv_0.$
 \medbreak
 Solving \eqref{eq:div} relies on  the following lemma:
 \begin{lemma}\label{lem:2} Assume that $\Omega$ is a $C^2$ bounded domain, the torus or the whole space. Fix $T>0$ and denote
 $$ E_T:=\Bigl\{ w\in\cC([0,T];L^2),\; \nabla w\in L^2(0,T\times\Omega), \!\!\quad w|_{\d\Omega}=0\!\andf\! w_t\in L^{4/3}(0,T\times\Omega)\Bigr\}\cdotp$$
 There exists a constant $c$ depending only on $\Omega$ such that whenever the divergence free vector-field $u$ satisfies
 \begin{equation}\label{eq:smallDv1}
 \|\nabla u\|_{L^2(0,T\times\Omega)} +\|\nabla u\|_{L^1(0,T;L^\infty)} \leq c,
 \end{equation}
 then, for all  vector-field $k\in\cC([0,T];L^2)$ such that $\div k\in L^2(0,T\times\Omega)$ and $k_t\in L^{4/3}(0,T\times\Omega),$
  there exists a  vector-field $w$ in the space $E_T$ satisfying 
  $$  \div(A w)=\div k,$$
  where $A$ is defined from $u$ as in \eqref{eq:Ai}, and the inequalities: 
  \begin{align}\label{eq:div1}
  \|w(t)\|_{L^2}&\leq C\|k(t)\|_{L^2}\quad\hbox{for all }\ t\in[0,T],\\
  \label{eq:div2}
  \|\nabla w\|_{L^2_T(L^2)} &\leq C\|\div k\|_{L^2_T(L^2)},\\
  \label{eq:div3} \|w_t\|_{L^{4/3}_T(L^{4/3})} &\leq C\bigl(\|k_t\|_{L^{4/3}_T(L^{4/3})}
  +\|\nabla u\|_{L^2_T(L^2)}\|w\|_{L^4_T(L^4)}\bigr)\cdotp\end{align}
   \end{lemma}
 \begin{proof}  With the notation of  Lemma \ref{lem:1} in Appendix, we introduce the map 
 $$\Phi:w\longmapsto z:=\cB\bigl( k+({\rm Id}-A)w\bigr)\cdotp$$ 
 It is only a matter of proving that $\Phi$ admits a fixed point. 
 That $\Phi$ maps $E_T$ to $E_T$ follows from Lemma \ref{lem:1} and easy modifications of the computations below. 
  Hence, as $E_T$ is a Banach space, it suffices to show that the linear map $\Phi$ is strictly contractive. 
  To do so, take two elements $w^1$ and $w^2$ of $E_T.$ 
  Then, we have 
  $$  \Phi(w^2)-\Phi(w^1)= \cB\bigl(({\rm Id}-A)\dw\bigr)\with  \dw:=w^2-w^1.$$
  Remembering \eqref{eq:Ai} and that $\cB:L^2\to L^2,$  we thus have
  \begin{equation}\label{eq:Phi1}\|\Phi(w^2)-\Phi(w^1)\|_{L^\infty_T(L^2)} \leq  C\|\nabla u\|_{L^1_T(L^\infty)}\|\dw\|_{L^\infty_T(L^2)}.
  \end{equation}
 Next, using again \eqref{eq:Ai} and the fact that 
 $$ \div\bigl(({\rm Id}-A)\dw\bigr)=\bigl({\rm Id}-A^\top\bigr):\nabla\dw,$$
 we readily get 
  \begin{equation}\label{eq:Phi2}\|\nabla(\Phi(w^2)-\Phi(w^1))\|_{L^2_T(L^2)} \leq  C\|\nabla u\|_{L^1_T(L^\infty)}\|\nabla\dw\|_{L^2_T(L^2)}.
  \end{equation}
  Finally, using that 
  $$  \bigl(({\rm Id}-A)\dw\bigr)_t=({\rm Id}-A)\dw_t - A_{t}\dw$$
 yields for a.e. $t\in[0,T],$ 
     \begin{align}\label{eq:Phi3}\|\bigl(\Phi(w^2)-\Phi(w^1)\bigr)_t(t)\|_{L^{4/3}} 
     &\lesssim\|({\rm Id}-A(t))\dw_t(t) \|_{L^{4/3}} +\|A_{t}(t)\dw(t)\|_{L^{4/3}}\nonumber\\
 &\lesssim\|\nabla u\|_{L_t^1(L^\infty)}\|\dw_t(t) \|_{L^{4/3}}\! +\!
  \|\nabla u(t)\|_{L^2}\|\dw(t)\|_{L^4}\nonumber\\
   \lesssim\|\nabla& u\|_{L_t^1(L^\infty)}\|\dw_t(t) \|_{L^{4/3}}
     +\|\nabla u(t)\|_{L^2}\|\dw(t)\|_{L^2}^{1/2}\|\nabla\dw(t)\|_{L^2}^{1/2}.
  \end{align}
 Putting \eqref{eq:Phi1}, \eqref{eq:Phi2} and \eqref{eq:Phi3} together, we conclude that 
 $$
 \|(\Phi(w^2)-\Phi(w^1)\|_{E_T}\leq C\bigl(\|\nabla u\|_{L^1_T(L^\infty)}+\|\nabla u\|_{L^2_T(L^2)}\bigr)\|\dw\|_{E_T}.
 $$
 Hence, if \eqref{eq:smallDv1} is satisfied with a suitable small $c>0$ then $\Phi$ is contractive, 
 which ensures the existence of $w$ in $E_T$ satisfying the desired equation.
 Finally, using the fact that we thus have
 $w =\cB k+\cB( (\Id-A)w),$ and that 
 $$ \div((\Id-A) w)=(\Id-A^\top):\nabla w\andf  ((\Id-A)w)_t=(\Id-A) w_t- A_t w,$$
 mimicking the above calculations gives \eqref{eq:div1}, \eqref{eq:div2} and \eqref{eq:div3}.
 \end{proof}
 In what follows, we assume that $T$ has been chosen so that \eqref{eq:smallDv1} is satisfied for $u^1$ and $u^2,$
  and we define $w$ on $[0,T]\times\Omega$  according to the above lemma with $k=-\dA\, v^2.$ 
  We shall use repeatedly that, owing to \eqref{def:dA} and Cauchy-Schwarz inequality,  we have
 \begin{equation}\label{eq:dA}
\max\Bigl( \| t^{-1/2} \dA\|_{L^\infty_T(L^2)}, \|(\dA)_t\|_{L^2(0,T\times\Omega)} \Bigr)\leq  \|\nabla\dv\|_{L^2(0,T\times\Omega)}.
 \end{equation}
 Hence, thanks to  \eqref{eq:div1},  we have  for all $t\in[0,T],$
\begin{equation}\label{eq:wL2}
\|w(t)\|_{L^2} \leq C\|\sqrt tv^2(t)\|_{L^\infty} \|\nabla\dv\|_{L^2(0,t\times \Omega)}.
\end{equation}
 Next, as 
 $$ (\dA v^2)_t=\dA_t v^2 + \dA\,v_{t}^2,$$
 Inequality \eqref{eq:div3} (before time integration)   and \eqref{def:dA} guarantee that
  \begin{equation}\label{eq:wt}
\|w_t\|_{L^{4/3}} \leq C\bigl(\|\nabla v^1\|_{L^2}\|w\|_{L^4} +\|\nabla\dv\|_{L^2}\|v^2\|_{L^4}
+\|\dA\|_{L^2}\|v_{t}^2\|_{L^4}\bigr)\cdotp
\end{equation}
 Finally,  using  $\div(\dA v^2)=\dA^\top:\nabla v^2,$  Inequalities \eqref{eq:div2} and  \eqref{eq:dA} yields
   \begin{equation}\label{eq:Dw}
 \|Dw(t)\|_{L^2}\leq  C\|\nabla\dv\|_{L^2_t(L^2)}\|\sqrt t\nabla v^2\|_{L^\infty_t(L^\infty)}.
 \end{equation}
    Now, taking the $L^2(0,t\times\Omega)$ scalar product of the first equation of \eqref{eq:w} 
  with $z$ and integrating by parts in some terms yields
 \begin{equation}\label{eq:z} \frac12\|\sqrt{\rho_0^1} z\|_{L^\infty(0,t;L^2)}^2+\int_0^t\|\nabla_{v^1}z\|_{L^2}^2\,d\tau 
  =\frac12 \|\sqrt{\rho_0^1}\,\du_0\|_{L^2}^2+ \sum_{j=1}^5 I_j(t)
 \end{equation}
 with 
 $$ \begin{aligned}
 I_1(t)&:=-\int_0^t\!\!\int_{\Omega} \bigl(\dA (A^2)^\top+A^1\dA^\top\bigr)\nabla v^2:\nabla z\,dx\,d\tau,\\
 I_2(t)&:=-\int_0^t\!\!\int_{\Omega}\dA^\top\nabla Q^2\cdot z\,dx\,d\tau,\\
 I_3(t)&:=-\int_0^t\!\!\int_{\Omega} \rho_0^1 w_\tau\cdot z\,dx\,d\tau,\\
 I_4(t)&:= -\int_0^t\!\!\int_{\Omega} (A^{1})^\top\nabla w :  (A^{1})^\top\nabla z\,dx\,d\tau,\\
  I_5(t)&:= -\int_0^t\!\!\int_{\Omega}  \dr_0\,v_{t}^2\cdotp z\,dx\,d\tau.\end{aligned} $$ 
 We shall often use  that, due to \eqref{eq:Ai}, we have
  \begin{equation}\label{eq:equiv}
  \|\nabla z\|_{L^2(0,T\times\Omega)}\simeq  \|\nabla_{v^1}z\|_{L^2(0,T\times\Omega)}.\end{equation}
  {}From this, we easily get
  $$I_1(t)\leq C\int_0^t\|\tau^{-\frac12}\dA(\tau)\|_{L^2} \|\sqrt\tau\nabla v^2(\tau)\|_{L^\infty}
  \|\nabla_{v^1}z(\tau)\|_{L^2}\,d\tau. $$
  Hence, using \eqref{eq:dA} 
  and Young inequality, 
  \begin{equation}\label{eq:II1}
  I_1\leq C\|\sqrt \tau \nabla v^2\|_{L^2_t(L^\infty)} ^2\|\nabla \dv\|_{L^2(0,t\times\Omega)}^2
 +\frac18\int_0^t \|\nabla_{v^1} z\|_{L^2}^2\,d\tau.
 \end{equation}
  Next, 
  by \eqref{eq:dA}, \eqref{eq:equiv}, H\"older inequality and \eqref{eq:lad}, we have
 $$\begin{aligned}
 I_2&\leq C  \int_0^t   \|\tau^{-1/2} \dA\|_{L^2} \|\sqrt \tau\nabla Q^2\|_{L^4}  \|z\|_{L^2}^{1/2}
 \|\nabla z\|_{L^2}^{1/2}\,d\tau,\\
  &\leq \frac18\int_0^t \|\nabla_{v^1} z\|_{L^2}^2\,d\tau+
  C\|\tau^{-1/2} \dA\|_{L^\infty_t(L^2)}^{4/3}   \|z\|_{L^\infty_t(L^2)}^{2/3}
    \int_0^t   \|\sqrt \tau\nabla Q^2\|_{L^4}^{4/3}\,d\tau.\end{aligned}$$
   Hence,  in light of \eqref{eq:dA}, of  Young inequality and of \eqref{eq:rho00}, we have 
  \begin{multline}\label{eq:I2}
I_2\leq  \frac18\int_0^t \bigl(\|\nabla_{v^1} z\|_{L^2}^2+ \frac14\|\nabla \dv\|_{L^2}^2\bigr)d\tau
+C  \|\sqrt{\rho_0^1}z\|_{L^\infty_t(L^2)}^{2}   \|\sqrt \tau\nabla Q^2\|_{L^{4/3}_t(L^4)}^{4}.
\end{multline}
In order to bound $I_3,$ we start with the inequality
$$
I_3\leq \rho^*\int_0^t \|w_\tau\|_{L^{4/3}}\|z\|_{L^4}\,d\tau.$$
Taking advantage of \eqref{eq:wt} to bound $w_\tau,$ and of  Gagliardo-Nirenberg and Young inequalities yields
$$\begin{aligned}
I_3&\lesssim \int_0^t \|z\|^{1/2}_{L^2}\|\nabla z\|^{1/2}_{L^2}
\Bigl(\|\nabla v^1\|_{L^2}\|w\|_{L^4}+\|v^2\|_{L^4}\|\nabla\dv\|_{L^2}
+\|\dA\|_{L^2}\|v_{\tau}^2\|_{L^4}\Bigr)d\tau\\ &\leq
\frac18\int_0^t\|\nabla_{v^1}z\|_{L^2}^2\,d\tau+\frac1{32}\int_0^t\|\nabla\dv\|_{L^2}^2\,d\tau
+C\int_0^t\|v^2\|_{L^4}^{4}\|z\|_{L^2}^{2}\,d\tau+I_{3,1}+I_{3,2}\\
&\with I_{31}:=C\int_0^t\|z\|_{L^2}^{2/3}\|\nabla v^1\|_{L^2}^{4/3}\|w\|_{L^2}^{2/3}\|\nabla w\|_{L^2}^{2/3}d\tau\\
&\andf  I_{32}:=C\int_0^t\|z\|_{L^2}^{2/3}\|\dA\|_{L^2}^{4/3}\|v_{\tau}^2\|_{L^4}^{4/3}d\tau.
\end{aligned}$$
Just using  \eqref{eq:dA} yields  
$$I_{32}\leq \|\nabla\dv \|_{L^2_t(L^2)}^{4/3}
 \|z\|_{L^\infty_t(L^2)}^{2/3}\|\sqrt\tau v_{\tau}^2\|_{L^{4/3}_t(L^4)}^{4/3}.$$
 In order to bound $I_{31},$ one has to use  \eqref{eq:wL2} and \eqref{eq:Dw}, which yields
$$\begin{aligned}
I_{31}&\leq C\int_0^t\|z\|_{L^2}^{2/3}\|\nabla v^1\|_{L^2}^{4/3}\|\sqrt \tau\,v^2\|_{L^\infty}^{2/3}\|\nabla\dv\|_{L^2_\tau(L^2)}^{2/3}
\|\tau^{-1/2}\dA(\tau)\|_{L^2}^{2/3}\|\sqrt\tau\nabla v^2\|_{L^\infty}^{2/3}\,d\tau\\
&\leq C\|\nabla\dv\|_{L^2_t(L^2)}^{4/3}\|z\|_{L^\infty_t(L^2)}^{2/3}
\int_0^t\|\sqrt \tau\,v^2\|_{L^\infty}^{2/3}\|\nabla v^1\|_{L^2}^{4/3}\|\sqrt\tau\nabla v^2\|_{L^\infty}^{2/3}\,d\tau.
\end{aligned}$$
This enables us to get the following bound for $I_3$:
   \begin{multline}\label{eq:II3}
   I_3(t) \leq   \frac18 \|\nabla_{v^1}z\|_{L^2_t(L^2)}^2 +\frac1{16}\|\nabla\dv\|_{L^2_t(L^2)}^2  +C\biggl(\|v^2\|_{L^4_t(L^4)}^4\\
  +\biggl(\int_0^t\|\sqrt \tau\,v^2\|_{L^\infty}^{2/3}\|\nabla v^1\|_{L^2}^{4/3}\|\sqrt\tau\nabla v^2\|_{L^\infty}^{2/3}\,d\tau\biggr)^3 
    +\|\sqrt\tau v_{\tau}^2\|_{L^{4/3}_t(L^4)}^{4}\biggr) \|\sqrt{\rho_0^1}z\|_{L^\infty_t(L^2)}^2.   
\end{multline}
Next,  thanks to \eqref{eq:Dw},  \eqref{eq:dA}, and Cauchy-Schwarz  and Young inequality,
\begin{align} I_4&\leq C\int_0^t \|\nabla w\|_{L^2}\|\nabla_{v^1}z\|_{L^2}\,d\tau\nonumber\\
&\leq C\int_0^t \|\tau^{-1/2}\dA\|_{L^2}\|\sqrt\tau\nabla v^2\|_{L^\infty} \|\nabla_{v^1}z\|_{L^2}\,d\tau,\nonumber\\\label{eq:I4}
&\leq  \frac18\int_0^t \|\nabla_{v^1} z\|_{L^2}^2\,d\tau+C\|\sqrt \tau \nabla v^2\|_{L^2(0,t;L^\infty)} ^2\|\nabla \dv\|_{L^2(0,t\times\Omega)}^2.\end{align}
Finally, it is obvious that 
   \begin{equation}\label{eq:I5}
I_5(t)\leq  \|\dr_0/\sqrt{\rho_0^1}\|_{L^\infty} \|\sqrt{\rho_0^1}z\|_{L^\infty_t(L^2)} \|v_{t}^2\|_{L^1_t(L^2)}.\end{equation}
So plugging \eqref{eq:II1}, \eqref{eq:I2}, \eqref{eq:II3}, \eqref{eq:I4}  and \eqref{eq:I5} in \eqref{eq:z} and taking $t=T$ yields 
$$\displaylines{\|\sqrt{\rho_0^1}\, z\|_{L^\infty_T(L^2)}^2+\|\nabla_{v^1}z\|_{L^2_T(L^2)}^2
\leq \|\sqrt{\rho_0^1}\,\du_0\|_{L^2}^2+ A(T) \|\sqrt{\rho_0^1}\,z\|_{L^\infty_T(L^2)}^{2}\hfill\cr\hfill
+ \biggl(\frac1{8}+C  \|\sqrt t \nabla v^2\|_{L^2_T(L^\infty)}^2\biggr)\|\nabla\dv\|_{L^2_T(L^2)}^2 
 +2\|\dr_0/\sqrt{\rho_0^1}\|_{L^\infty}^2  \|v_{t}^2\|_{L^1_T(L^2)}^2\cr
\with A(T):= C\biggl(\|v^2\|_{L^4_T(L^4)}^4+\|\sqrt t v_{t}^2\|_{L^{4/3}_T(L^4)}^4 + \|\sqrt \tau\nabla Q^2\|_{L^{4/3}_T(L^4)}^{4}
\hfill\cr\hfill
+\biggl(\int_0^t\|\sqrt \tau\,v^2\|_{L^\infty}^{2/3}\|\nabla v^1\|_{L^2}^{4/3}\|\sqrt\tau\nabla v^2\|_{L^\infty}^{2/3}\,d\tau\biggr)^3\biggr)\cdotp}$$
The regularity properties of the constructed solutions guarantee that $A(\infty)$ is finite, 
and Lebesgue dominated convergence theorem thus ensures that  if $T$ is small enough, then
\begin{equation}\label{eq:T}\max\bigl(8C\|\sqrt t \nabla v^2\|_{L^2_T(L^\infty)}^2, 2A(T)\bigr) \leq1.\end{equation}
Under this hypothesis, the  above inequality becomes
\begin{multline}\label{eq:zz}
\frac12\|\sqrt{\rho_0^1}\, z\|_{L^\infty_T(L^2)}^2+\|\nabla_{v^1}z\|_{L^2_T(L^2)}^2
\\\leq   \|\sqrt{\rho_0^1}\,\du_0\|_{L^2}^2+  \frac1{4}\|\nabla \dv\|_{L^2_T(L^2)}^2 + C\|\dr_0\|_{L^\infty}^2  \|v_{t}^2\|_{L^1_T(L^2)}^2.
\end{multline}
Since $\nabla\dv =\nabla z+\nabla w,$ we may write
 owing to \eqref{eq:dA}, \eqref{eq:Dw} and \eqref{eq:equiv}, 
 $$\begin{aligned}
 \|\nabla \dv\|_{L^2_T(L^2)}^2&\leq 2\|\nabla z\|_{L^2_T(L^2)}^2
+ 2\|\nabla w\|_{L^2_T(L^2)}^2\\
&\leq \frac52\|\nabla_{v^1} z\|_{L^2_T(L^2)}^2s
+ C  \|\sqrt t \nabla v^2\|_{L^2_T(L^\infty)}^2
 \|\nabla\dv\|_{L^2_T(L^2)}^2.\end{aligned}$$
 Hence, under assumption \eqref{eq:T} (up to a change of $C$ if needed), we have
  \begin{equation}\label{eq:nabladv}
 \|\nabla \dv\|_{L^2(0,T\times\Omega)}^2 \leq 3\|\nabla_{v^1} z\|_{L^2(0,T\times\Omega)}^2.\end{equation}
Plugging this inequality in \eqref{eq:zz} gives
\begin{equation}\label{eq:stab0}\frac12\|\sqrt{\rho_0^1}\, z\|^2_{L^\infty_T(L^2)}+\frac14\|\nabla_{v^1}z\|^2_{L^2_T(L^2)}
\leq C\Bigl(\|\sqrt{\rho_0^1}\,\du_0\|_{L^2}^2+\|\dr_0\|_{L^\infty}^2 \|v_{t}^2\|_{L^1_T(L^2)}^2\Bigr)\cdotp\end{equation}
In the case where the two solutions correspond to the same initial data,  this ensures
that $z\equiv0$ on $[0,T].$ Then, remembering \eqref{eq:nabladv} and \eqref{eq:wL2}, 
one can conclude to uniqueness on $[0,T],$ then on $\R_+$ by standard bootstrap.


\subsection{Continuity of the flow map}

Here we consider the case where the two solutions considered in the previous paragraph correspond to possibly  different data. 
As a first, we have to observe that  \eqref{eq:nabladv} and \eqref{eq:stab0} together imply that if 
\begin{equation}\label{eq:sqrtv2}
\|\sqrt tv^2\|_{L^\infty(\R_+\times\Omega)} \leq K,
\end{equation}
then, in light of Inequalities \eqref{eq:wL2}, \eqref{eq:nabladv} and \eqref{eq:stab0}, there exists some constant $c>0$ 
such that if $\wt A(T_0)\leq c,$ then we have
\begin{equation}\label{eq:stab1}
\|\sqrt{\rho_0^1}\, \dv\|_{L^\infty_{T_0}(L^2)}+\|\nabla_{v^1}\dv\|_{L^2_{T_0}(L^2)}
\leq C(1+K)\Bigl(\|\sqrt{\rho_0^1}\,\du_0\|_{L^2}+ \|\dr_0\|_{L^\infty}\Bigr)
\end{equation}
where we have  denoted for all $T\in[0,\infty]$:
$$\displaylines{\wt A(T):= \|v^2\|_{L^4_T(L^4)}^{4}+\|\sqrt t(v_{t}^2,\nabla Q^2)\|_{L^{4/3}_T(L^4)}^{4/3}
\hfill\cr\hfill
+(1+K)\bigl(\|\nabla v^1\|_{L^2_T(L^2)}^{2}+\|\sqrt\tau\nabla v^2\|_{L^2_T(L^\infty)}^{2}\bigr) +  \|v_{t}^2\|_{L^1_T(L^2)}.}$$
Now, if we consider data that belong to a bounded subset of $\wt B^0_{\rho_0,1}$ then $K$ in \eqref{eq:sqrtv2} 
and $\wt A(\infty)$ can be uniformly bounded. 
By iterating the procedure that led to \eqref{eq:stab1}, this allows to get in the end 
\begin{equation}\label{eq:stabv}\|\sqrt{\rho_{0}^1}\, \dv\|_{L^\infty_T(L^2)}+\|\nabla_{v^1}\dv\|_{L^2_T(L^2)}
\leq Ce^{C\wt A(\infty)}\Bigl(\|\sqrt{\rho_0^1}\,\du_0\|_{L^2}+ \|\dr_0\|_{L^\infty}\Bigr)\cdotp\end{equation}
Then, reverting to the  Eulerian coordinates gives the following stability statement:
\begin{theorem}\label{thm:3} Consider two solutions $(\rho^1,u^1,P^1)$ and $(\rho^2,u^2,P^2)$ 
corresponding to initial data $(\rho_0^1,u_0^1)$ and $(\rho_0^2,u_0^2)$ given by Theorem \ref{thm:2}. 
Assume that $$0<\rho_*\leq \rho_0^1,\rho_0^2\leq \rho^*\andf 
\max\bigl(\|u_0^1\|_{\wt B^0_{\rho_0^1,1}},\|u_0^2\|_{\wt B^0_{\rho_0^2,1}}\bigr)\leq M.$$
Then we have:
\begin{equation}\label{eq:stabu}
\|\sqrt{\rho_{0}^1}\, \du\|_{L^\infty_T(L^2)}+\|\nabla\du\|_{L^2_T(L^2)}\leq C_{\rho_*,\rho^*,M}\Bigl(\|\sqrt{\rho_0^1}\,\du_0\|_{L^2}+ \|\dr_0\|_{L^\infty}\Bigr),
\end{equation}
and, for all $p\in[2,\infty),$ 
\begin{equation}\label{eq:stabrho}
\|\dr(t)\|_{\dot W^{-1,p}}  \leq C_{p,\rho_*,\rho^*,M}
\Bigl(\|\dr_0\|_{\dot W^{-1,p}} + t^{\frac12+\frac1p}  \bigl( \|\sqrt{\rho_0^1}\,\du_0\|_{L^2}+ \|\dr_0\|_{L^\infty}\bigr)\Bigr)\cdotp
\end{equation}
\end{theorem}
\begin{proof}
Although our regularity assumptions  are weaker,  we shall follow \cite{DMP} to bound the difference of the velocities. The starting point is the relation:
    $$\begin{aligned}\nabla_y\dv= K_1\!+\!K_2\!+\!K_3\with
    K_1(t,y)&:= \nabla_y\dX(t,y)\cdot\nabla_xu^2(t,X^2(t,y)),\\
    K_2(t,y)&:=  \nabla_yX^1(t,y)\cdot\nabla_x\du(t,X^2(t,y))\\
    \andf K_3(t,y):= \nabla_yX^1&(t,y)\cdot\bigl(\nabla_x u^1(t,X^2(t,y))-\nabla_x u^1(t,X^2(t,y))\bigr)\cdotp
    \end{aligned}$$
Since $\nabla\du(t,X^2(t,y))=A_1^\top(t,y) K_2(t,y)$ and  the flow $X^2$  is measure preserving,  the above decomposition 
implies that
$$\|\nabla\du\|_{L^2}\leq \|A_1\|_{L^\infty}\bigl(\|\nabla\dv\|_{L^2}+\|K_1\|_{L^2}+\|K_3\|_{L^2}\bigr)\cdotp$$
Bounding $K_1$ may be done as in \cite{DMP}. We get for all $t\geq0,$
$$\|K_1(t)\|_{L^2}\leq C\|\sqrt t\nabla u^2(t)\|_{L^\infty} \|\nabla\dv\|_{L^2_t(L^2)}.$$
For bounding $K_3,$ we use the relation
$$K_3(t,y)=\nabla X_1(t,y)\cdot\biggl(\int_1^2\bigl(\nabla^2 u^1(t,X^s(t,y))\bigr)\cdotp\Bigl(\frac{dX^s}{ds}(t,y)\Bigr)ds\biggr)$$
where the `interpolating flow' $X^s$  stands for the solution to 
$$X^s(t,y)=y+\int_0^t\bigl((2-s)u^1(\tau,X^s(\tau,y))+(s-1)u^2(\tau,X^s(\tau,y))\bigr)d\tau.$$
As $X^s(t,\cdot)$  is also measure preserving, it is easy to prove that (again, see \cite{DMP}):
$$\biggl\|\frac{dX^s}{ds}(t,\cdot)\biggr\|_{L^4}\leq C \|\du\|_{L^1_t(L^4)}.$$
Thanks to that and to H\"older inequality, we deduce that 
$$\|K_3(t)\|_{L^2}\leq C\bigl(1+\|\nabla u^1\|_{L^1_t(L^\infty)}) \|t^{3/4}\nabla^2 u^1(t)\|_{L^4}\|\du\|_{L^4_t(L^4)}.$$
Hence, in the end, if $T$ is chosen so that 
$$\max\biggl(\int_0^T\|\nabla u^1(t)\|_{L^\infty}\,dt,\; \int_0^T\|\nabla u^2(t)\|_{L^\infty}\,dt\biggr)\leq1,$$
then we have, using also \eqref{eq:stokesLp}
$$\|\nabla\du\|_{L^2_T(L^2)}\lesssim \bigl(1+\|\sqrt t\nabla u^2\|_{L_T^2(L^\infty)}\bigr)\|\nabla\dv\|_{L^2_T(L^2)}
+\|t^{3/4}\dot u^1\|_{L^2_T(L^4)} \|\du\|_{L^4_T(L^4)}.$$
The last term may be handled by means of \eqref{eq:lad}, and one ends up with
\begin{multline}\label{eq:nabladu}
\|\nabla\du\|_{L^2_T(L^2)}\lesssim \bigl(1+\|\sqrt t\nabla u^2\|_{L_T^2(L^\infty)}\bigr)\|\nabla\dv\|_{L^2_T(L^2)}
\\+\|t^{3/4}\dot u^1\|_{L^2_T(L^4)}^2 \|\sqrt{\rho^1}\du\|_{L^\infty_T(L^2)}.\end{multline}
Remember that the constructed solutions satisfy  $\sqrt t\nabla u^2\in L^2(\R_+;L^\infty)$ and note that, since 
$$\|t^{3/4}\dot u^1\|_{L^2_T(L^4)} \leq C \|t\dot u^1\|_{L^\infty_T(L^2)}^{1/2} \|\sqrt t D\dot u^1\|_{L^1_T(L^2)}^{1/2}, $$
Inequalities \eqref{eq:weighttt} and \eqref{eq:tuL1L1b} guarantee that  $t^{3/4}\dot u^1$ is in $L^2(\R_+;L^4).$ 
So we are left with bounding $\sqrt\rho^1 \du$  in $L^\infty(0,T;L^2).$  To do so, we use, as in \cite{DMP} the following relation:
$$\sqrt{\rho_{0}^1(y)}\dv(t,y)=\sqrt{\rho^1}(t,X^1(t,y))\biggl(\du(t,X^1(t,y))+\int_1^2 \!\!\!Du^2(t,X^s(t,y))\frac{dX^s}{ds}(t,y)\,ds\biggr)\cdotp$$
Hence, as all the flows $X^s$  are measure preserving and $\rho_1$ is bounded from below, 
$$\begin{aligned}
\|\sqrt{\rho^1(t)}\du(t)\|_{L^2}&\leq \|\sqrt{\rho_{0}^1}\dv(t)\|_{L^2} +C\sqrt{\rho^*}\|Du^2(t)\|_{L^4}\|\du\|_{L^1_t(L^4)}\\
&\leq \|\sqrt{\rho_{0}^1}\dv(t)\|_{L^2} +C\|t^{3/4}Du^2(t)\|_{L^4}\|\du\|_{L^4_t(L^4)}\\
&\leq \|\sqrt{\rho_{0}^1}\dv(t)\|_{L^2}\\&\quad +C\|\sqrt tDu^2(t)\|_{L^2}^{1/2}\|tD^2u^2(t)\|_{L^2}^{1/2}\|\nabla\du\|_{L^2_t(L^2)}^{1/2}
\|\sqrt{\rho^1(t)}\du\|_{L^\infty_t(L^2)}^{1/2}.
\end{aligned}$$
Since both the terms with $\sqrt t Du^2$ and with $t D^2u^2$ may be bounded in terms of $\rho_*,$ $\rho^*$ and 
$\|u_{0}^2\|_{L^2}$ only, 
we end up with 
$$\|\sqrt{\rho^1}\,\du\|_{L^\infty_T(L^2)}\leq 2 \|\sqrt{\rho_{0}^1}\,\dv\|_{L^\infty_T(L^2)} +C(\rho_*,\rho^*,\|u_{0}^2\|_{L^2})\|\nabla\du\|_{L^2_T(L^2)}.
$$
Putting this inequality together with \eqref{eq:nabladu} and remembering \eqref{eq:stabv} allows to conclude that there exists an absolute
constant $C$ such that for small  enough $T,$ we have
$$
\|\sqrt{\rho_{0}^1}\, \du\|_{L^\infty_T(L^2)}+\|\nabla\du\|_{L^2_T(L^2)}
\leq C\Bigl(\|\sqrt{\rho_0^1}\,\du_0\|_{L^2}+ \|\dr_0\|_{L^\infty}\Bigr),$$
then arguing by induction and using the bounds on $u^1$ and $u^2$ in terms of the data yields \eqref{eq:stabu}.  
\smallbreak
Finally, the difference between the (Eulerian) densities may be bounded by resorting to the classical theory of transport equation. 
Indeed, we have
$$\d_t\dr+\div(\dr\,u^2) =-\div(\rho^1 \du).$$
Hence, we may write for all $p\in[1,\infty]$ and  $t\geq0,$
$$\begin{aligned}\|\dr(t)\|_{\dot W^{-1,p}} &\leq\biggl(\|\dr_0\|_{\dot W^{-1,p}} + \int_0^t e^{-\int_0^\tau\|\nabla u^2\|_{L^\infty}\,d\tau'}\|\rho_1\du\|_{L^p}\,d\tau\biggr)e^{\int_0^t\|\nabla u^2\|_{L^\infty}\,d\tau}\\
&\leq\biggl(\|\dr_0\|_{\dot W^{-1,p}} + \rho^* t^{\frac12+\frac1p}  \|\du\|_{L^{\frac{2p}{p-2}}_t(L^p)}\biggr)e^{\int_0^t\|\nabla u^2\|_{L^\infty}\,d\tau}.
\end{aligned}$$
Combining Inequality  \eqref{eq:stabu} with Gagliardo-Nirenberg inequality provides us with a control of $\du$ in 
$L^{\frac{2p}{p-2}}(\R_+;L^p)$ for all $p\in[2,\infty).$ In the end, we get \eqref{eq:stabrho}. 
\end{proof}
\begin{remark}
In the bounded or torus cases, one can  take advantage of exponential decay to get a time independent 
bound. The details are left to the reader.
\end{remark}
	
	
\section{Appendix} 

Here we recall  some results that played a key role throughout the paper.
The first one is the following  Gagliardo-Nirenberg inequality that extends \eqref{eq:lad}: 
\begin{equation}\label{eq:GN}
\|z\|_{L^p}\leq C_p \|z\|_{L^2}^{2/p}\|\nabla z\|_{L^2}^{1-2/p}, \qquad 2\leq p<\infty.\end{equation}
It holds true with  the same constant in $\R^2$ and  for any   $z\in H^1_0(\Omega)$  
in a general domain $\Omega,$ or in the torus $\T^2$ \emph{provided the mean value of $z$ is zero.} 
In the torus case however, we rather are  in situations where 
$$\int_{\T^2} a z\,dx=0$$ 
for some nonnegative measurable function $a$ with  positive mean value (say $1$ with no loss of generality).
Then,    we claim that
\begin{equation}\label{eq:GNT}\|z\|_{L^p}\leq C_{p,a} \|z\|_{L^2}^{2/p}\|\nabla z\|_{L^2}^{1-2/p}
\with C_{p,a}:= C_p\log^{\frac{p-2}p}\bigl(e+\|a\|_{L^2}\bigr)\cdotp\end{equation}
Indeed, decomposing $z$ into $z=\bar z+\wt z$ with $\bar z:=\int_{\T^2} z\,dx,$  we have:
$$\begin{aligned}
\int_{\T^2} |z|^p\,dx&=\int_{\T^2} |z|^2|\wt z +\bar z|^{p-2}\,dx\\
&\lesssim  |\bar z|^{p-2}\|z\|_{L^2}^2+ \int_{\T^2}|z|^2|\wt z|^{p-2}\,dx\\
&\lesssim  |\bar z|^{p-2}\|z\|_{L^2}^2+ \|z\|_{L^p}^2\|\wt z\|^{p-2}_{L^p}.\end{aligned}$$
Now,  $\wt z$ is mean free and thus satisfies \eqref{eq:GN}.  Besides,  according to \cite[Ineq. (A.2)]{DM1},
$$|\bar z|\leq C\log\bigl(e+\|a\|_{L^2}\bigr)\|\nabla z\|_{L^2}.$$
 Hence
$$\|z\|_{L^p}^p\leq  C\log\bigl(e+\|a\|_{L^2}\bigr)\|\nabla z\|_{L^2}^{p-2}\|z\|_{L^2}^2 + C_p\|z\|_{L^p}^2
 \bigl(\|\wt z\|_{L^2}^{2/p}\|\nabla z\|_{L^2}^{1-2/p})^{p-2}.$$
 Then, \eqref{eq:GNT} follows from    $\|\wt z\|_{L^2}\leq \|z\|_{L^2}.$\qed
\medbreak
Next, we recall a well known  result for the inhomogeneous Stokes equations:
	\begin{equation}\label{eq:stokesinhomo}-\Delta w+\nabla Q=f\andf \div w=g\qquad \hbox{in }\ \Omega\end{equation}
with data $f\in L^p(\Omega)$ and $g\in \dot W^{1,p}(\Omega),$ $1<p<\infty.$
\smallbreak
In the bounded domain case (with $g$ having mean value $0$), it is known (see e.g. \cite{Galdi}) that \eqref{eq:stokesinhomo}
admits a unique solution $(w,\nabla Q)\in W^{2,p}(\Omega)\times L^p(\Omega)$ such that $w|_{\d\Omega}=0,$ and that the following 
bound holds true:
\begin{equation}\label{eq:stokesLp}\|\nabla^2w,\nabla Q\|_{L^p}\leq C\bigl(\|f\|_{L^p}+\|\nabla g\|_{L^p}\bigr)\cdotp\end{equation}
A similar result holds true in  $\Omega=\R^2$ or $\Omega=\T^2$  provided we consider only solutions 
such that $w\to0$ at infinity ($\R^2$ case) or $\int_{\T^2} a w\,dx=0$ 
for some nonnegative bounded  function $a,$ with mean value $1$ (torus case).  
Indeed: one can set 
$$\nabla Q=\cQ f\with\cQ:=-(-\Delta)^{-1}\nabla\div,$$
then solve the Poisson equation 
$-\Delta w=f+\nabla Q.$  Uniqueness is given by the supplementary conditions that are prescribed above. 
	\medbreak
	Finally, in the proof of stability and uniqueness, we used
	the following result. 
	 \begin{lemma}\label{lem:1} Assume that $\Omega$ is a $C^2$ bounded domain, the torus or the whole space. Then, there exists 
 a linear operator $\cB$ that maps $L^p$ to $L^p$ for all $p\in(1,\infty)$ such that 
 for all  $k\in L^p(\Omega;\R^d)$ (with mean value $0$ in the case $\Omega=\T^d$) we have  
 $$\div(\cB k)=\div k.$$ 
 Furthermore, if $\div k\in L^q(\Omega)$ for some $q\in(1,\infty),$ then we have $\cB k\in W^{1,q}_0(\Omega;\R^n)$ 
 with $\|\nabla\cB k\|_{L^q}\leq C \|\div k\|_{L^q}$ and if $k$  (seen as a function from $\R_+$ 
 to some space $L^r$ with $1<r<\infty$) 
 is differentiable for almost every $t\in\R_+,$
 then so does $\cB k,$ and  we have   $\|(\cB k)_t\|_{L^r}\leq C \|k_t\|_{L^r}$  for a.e. $t\in\R_+.$ 
 \end{lemma}
 \begin{proof}
 Whenever $\Omega$ is a $C^2$ bounded domain, the existence of $\cB$ as well as the first two properties have been established in \cite{DM-div}. 
 The third one stems from the fact that, owing to the continuity and linearity of $\cB,$ we may write in the $L^r$ meaning that
 $$ (\cB k)_t(t)=\lim_{h\to0} \frac{\cB k(t+h)-\cB k(t)}h=\lim_{h\to0} \cB\biggl(\frac{k(t+h)-k(t)}h\biggr)
 =\cB k_t. $$
 If $\Omega$ is the torus or the whole space, then one can just set $\cB:=-(-\Delta)^{-1}\nabla\div.$
  \end{proof}

   \begin{small}	 

\end{small}

\bigbreak

\noindent\textsc{Univ Paris Est Creteil, Univ Gustave Eiffel, CNRS, LAMA UMR8050, F-94010 Creteil, France
and Sorbonne Universit\'e, LJLL UMR 7598, 4 Place Jussieu, 75005 Paris}\par\nopagebreak
E-mail address: danchin@u-pec.fr

	\end{document}